\documentclass[
preprint, 3p, 
number, 
sort&compress,
]{elsarticle}
\pdfoutput=1

\usepackage[utf8]{luainputenc}
\usepackage[english]{babel}
\usepackage{csquotes}

\usepackage{hyperref}

\usepackage{amsmath}
\allowdisplaybreaks
\usepackage{amssymb}
\usepackage{commath}
\usepackage{mathtools}
\usepackage{bbm}

\usepackage{siunitx}
\usepackage{adjustbox}

\usepackage{amsthm}
\theoremstyle{plain}
  \newtheorem{theorem}{Theorem} 
   
	\newtheorem{proposition}[theorem]{Proposition}
\theoremstyle{definition}
  \newtheorem{definition}[theorem]{Definition}
  \newtheorem{remark}[theorem]{Remark}

\usepackage{color}
\usepackage{graphicx}
\usepackage[small]{caption}
\usepackage{subcaption}

\ifx\useTikzForPlotting\undefined
\else
  \usepackage{pgfplots}
  \pgfplotsset{compat=1.11}
  \usetikzlibrary{external}
\fi

\usepackage{booktabs}
\usepackage{rotating}
\usepackage{multirow}

\usepackage{multicol}
\usepackage{enumitem}

\usepackage{calc}
\usepackage{xparse}

\DeclareMathOperator*{\argmin}{arg\,min}
\renewcommand{\vec}[1]{\underline{#1}}
\NewDocumentCommand{\mat}{mo}{%
  \IfValueTF{#2}{%
    \underline{\underline{#1}}{#2}
  }{%
    \underline{\underline{#1}}\,
  }%
}

\newcommand{\sign}{\operatorname{sign}}
\newcommand{\scp}[2]{\left\langle{#1,\, #2}\right\rangle}

\renewcommand{\d}{\mathrm{d}} 
\newcommand{\intd}{\, \mathrm{d}}

\renewcommand{\epsilon}{\varepsilon}
\renewcommand{\phi}{\varphi}
\renewcommand{\rho}{\varrho}
\newcommand{\N}{\mathbb{N}}
\newcommand{\R}{\mathbb{R}}

\newsavebox{\DelimiterBox}
\newlength{\DelimiterHeight}
\newlength{\DelimiterDepth}
\newsavebox{\ArgumentBox}
\newlength{\ArgumentHeight}
\newlength{\ArgumentDepth}
\newlength{\ResizedDelimiterHeight}
\newlength{\ResizedDelimiterDepth}

\usepackage{lipsum}
\makeatletter
\def\ps@pprintTitle{%
 \let\@oddhead\@empty
 \let\@evenhead\@empty
 \def\@oddfoot{}%
 \let\@evenfoot\@oddfoot}
\makeatother

\begin{document}

\begin{frontmatter}

\title{Stable High Order Quadrature Rules for Scattered Data and General Weight Functions}

\author[label1,label2]{Jan Glaubitz\corref{cor1}}
\ead{Jan.Glaubitz@Dartmouth.edu}

\cortext[cor1]{Corresponding author: Jan Glaubitz}

\address[label1]{Max Planck Institute for Mathematics, Vivatsgasse 7, 53111 Bonn, Germany.}
\address[label2]{Department of Mathematics, Dartmouth College, 27 North Main Street, Hanover, NH 03755, USA.}

\begin{abstract}
  Numerical integration is encountered in all fields of numerical analysis and the engineering sciences. 
By now, various efficient and accurate quadrature rules are known; for instance, Gauss-type quadrature rules. 
In many applications, however, it might be impractical---if not even impossible---to obtain data to fit known 
quadrature rules. 
Often, experimental measurements are performed at equidistant or even scattered points in space or time. 
In this work, we propose stable high order quadrature rules for experimental data, which can accurately handle general 
weight functions. 

\end{abstract}


\end{frontmatter}

\section{Introduction} 
\label{sec:introduction} 

An omnipresent problem in mathematics and the applied sciences is the computation of integrals of the form 
\begin{equation}\label{eq:I}
  I[f] := \int_a^b f(x) \omega(x) \intd x 
\end{equation}
with \textit{weight function} ${\omega:[a,b] \to \R}$. 
In many situations, we are faced with the problem to recover $I[f]$ from a finite set of measurements 
$\{f(x_n)\}_{n=1}^N \subset \R$ at distinct points ${\{x_n\}_{n=1}^N \subset [a,b]}$. 
This problem is referred to as \textit{numerical integration}. 
A universal approach is to make use of the data $\{(x_n,f(x_n))\}_{n=1}^N$ and to approximate the integral \eqref{eq:I} 
by a finite sum 
\begin{equation}
  Q_N[f] := \sum_{n=1}^N \omega_n f(x_n)
\end{equation}
which is called an \textit{$N$-point quadrature rule (QR)}. 
The points $\{x_n\}_{n=1}^N$ are referred to as the \textit{quadrature points}, and the $\{\omega_n\}_{n=1}^N$ are 
referred to as the \textit{quadrature weights}. 
A QR is uniquely determined by its quadrature points and quadrature weights. 

\subsection{State of the art}

In the case of nonnegative weight functions ${\omega}$, many QRs with favorable 
properties are known. 
We refer to a rich body of literature 
\cite{gautschi1997numerical,krylov2006approximate,davis2007methods,brass2011quadrature}. 
W.\,r.\,t.\ their quadrature points, we can loosely differentiate between two classes of QRs: 

\begin{enumerate}
  \item QRs for which the quadrature points can be chosen freely. 
  
  \item QRs for which the quadrature points are prescribed, for instance, by a set of equidistant or 
scattered points.
\end{enumerate}

In the first class of QRs, the quadrature points and weights are typically determined such that the degree 
of exactness is maximized. 
We say an $N$-point QR $Q_N$ has \textit{(algebraic) degree of exactness $d$} when 
\begin{equation}\label{eq:exactness-cond}
    Q_N[p] = I[p] \quad \forall p \in \mathbb{P}_d([a,b])
\end{equation} 
is satisfied, where $\mathbb{P}_d([a,b])$ is the linear space of all (algebraic) polynomials on $[a,b]$ of degree at 
most $d$. 
Examples of such QRs are Gauss and Gauss--Lobatto-type QRs 
\cite[Chapter 25.4]{abramowitz1964handbook}. 
Further, we mention Fej\'er \cite{fejer1933infinite} and Clenshaw--Curtis \cite{clenshaw1960method} QRs. 
See the work of Trefethen \cite{trefethen2008gauss} for an interesting comparison between Clenshaw--Curtis and 
Gauss-type QRs. 

Yet, in many applications, it may be impractical---if not even impossible---to obtain data to fit 
these QRs. 
Typically, experimental measurements are performed at equidistant or scattered points. 
In this situation, we are in need of QRs of the second class, where the quadrature points are 
prescribed and we are faced with the task to find suitable quadrature weights. 
Composite Newton--Cotes rules for equidistant points and the composite midpoint as well as the composite trapezoidal 
QR for scattered points are first simple examples, yet they are of limited degree of exactness. 
It is well known that Newton--Cotes rules for equidistant points start to become unstable for  
degrees of exactness higher than $d=10$. 
This instability might even intensify for more general interpolatory QRs for 
scattered quadrature points. 
Stable QRs on equidistant and scattered points with high degrees of exactness have been proposed by Wilson 
\cite{wilson1970discrete,wilson1970necessary} 
and revisited in \cite{huybrechs2009stable} and \cite[Chapter 4]{glaubitz2020shock}. 
Moreover, they have been applied to construct stable high order (discontinuous Galerkin) methods for partial 
differential equations in \cite{glaubitz2020stable}.

All the above QRs only apply to integrals of the form \eqref{eq:I} with nonnegative weight functions. 
Yet, in many situations we have to deal with weight functions with mixed signs. 
For instance, in the weak form of the Schr\"odinger equation, integrals of the form 
\begin{equation}
  I[f] = \int_a^b f(x) V(x) \intd x
\end{equation} 
arise, where $V$ is a potential with possibly mixed signs. 
Another field of application is (highly) oscillatory integrals of the form 
\begin{equation}
  I[f] = \int_a^b f(x) e^{i \omega g(x)} \intd x, 
\end{equation}
where $g$ is a smooth oscillator and $\omega$ is a frequency parameter ranging from $0$ to $\infty$. 
See \cite{iserles2004quadrature,iserles2006highly,huybrechs2009highly} and references therein.

\subsection{Our contribution} 

In this work, we investigate stability concepts for QRs approximating integrals \eqref{eq:I} with 
general weight functions $\omega$. 
For nonnegative weight functions stability of (a sequence of) QRs essentially follows from QRs having 
non\-negative-only 
quadrature weights. 
For general weight functions, however, this connection breaks down, and we have to consider stable and 
sign-consistent (the quadrature weights have the same sign as the weight function at the quadrature points) QRs as 
separated classes. 
Following these observations, we propose two different procedures to construct stable QRs as well as sign-consistent 
QRs with high degrees of exactness 
on equidistant and scattered quadrature points for general weight functions. 
The procedure to construct stable (sequences of) QRs with high degrees 
of exactness builds upon a least squares (LS) formulation of the exactness conditions \eqref{eq:exactness-cond}.
The procedure to construct sign-consistent QRs with high degrees of exactness builds upon a nonnegative LS (NNLS) formulation of the exactness conditions \eqref{eq:exactness-cond}. 
Both procedures have already been investigated by Huybrechs \cite{huybrechs2009stable} for positive weight functions. 
Here, we modify the procedures to handle general weight functions and present a discussion of their stability. 

\subsection{Outline}

The rest of this work is organized as follows.
After this introduction, \S \ref{sec:stability} revisits the concept of stability for QRs. 
In particular, we introduce the concept of sign-consistency. 
In \S \ref{sec:LS-QRs}, we construct stable (sequences of) QRs with high degrees 
of exactness by a LS formulation of the exactness conditions \eqref{eq:exactness-cond}. 
Sign-consistent QRs with high degrees of exactness, on the other hand, are constructed in \S 
\ref{sec:NNLS-QRs} by an NNLS formulation of the exactness conditions \eqref{eq:exactness-cond}. 
Finally, stable and sign-consistent (sequences of) QRs are compared by numerical tests in \S 
\ref{sec:tests}. 
We end this work with some concluding thoughts in \S \ref{sec:summary}. 

\section{Stability of QRs} 
\label{sec:stability} 

Besides exactness \eqref{eq:exactness-cond}, stability is a crucial property of QRs as well as of any 
other mathematical problem. 
Speaking heuristically, stability ensures that small changes in the input argument only yield small changes in the 
output value. 
See \cite[Lecture 14]{trefethen1997numerical} for a more detailed discussion on stability.

Let us consider functions $f$ and $\tilde{f}$ with ${| f(x) - \tilde{f}(x) | \leq \varepsilon}$ for $x \in [a,b]$. 
Then, we have 
\begin{equation}\label{eq:stab-I}
    \left| I[f] - I[\tilde{f}] \right| 
        \leq \int_a^b \left| \left( f(x) - \tilde{f}(x) \right) \omega(x) \right| \intd x 
        \leq \varepsilon K_\omega 
\end{equation} 
for the weighted integral $I$, where we define 
\begin{equation}\label{eq:K-omega}
    K_\omega := \int_a^b |\omega(x)| \intd x.
\end{equation} 
Thus, the growth of errors in the integral $I$ is bounded by the factor $K_\omega$. 
A similar behavior is desired for QRs approximating $I$. 
Denoting the vector of weights $(\omega_1,\dots,\omega_N)^T$ by $\vec{\omega}_N$, we have 
\begin{equation}\label{eq:stab-Q}
    \left| Q_N[f] - Q_N[\tilde{f}] \right| 
        \leq \sum_{n=1}^N \left| \omega_n \left( f(x_n) - \tilde{f}(x_n) \right) \right| \intd x 
        \leq \varepsilon \kappa(\vec{\omega}_N)
\end{equation} 
for the $N$-point QR $Q_N$, where we define 
\begin{equation}
    \kappa(\vec{\omega}_N) := \sum_{n=1}^N \left| \omega_n \right|.
\end{equation} 
$\kappa(\vec{\omega}_N)$ is a common stability measure for QRs. 
Unfortunately, the factor $\kappa(\vec{\omega}_N)$, which bounds the propagation of errors, might grow for increasing 
$N$. 
Hence, the QR might become less accurate even though more and more data is used. 
This problem is encountered by the concept of stability. 

\begin{definition}\label{def:stability}
  We call (a sequence of) $N$-point QRs $(Q_N)_{N \in \N}$ with quadrature weights ${(\vec{\omega}_N)_{N 
\in \N}}$ \emph{stable} if $\kappa(\vec{\omega}_N)$ is uniformly bounded w.\,r.\,t.\ $N$, i.\,e., if 
  \begin{equation}
    \sup_{N \in \N} \kappa(\vec{\omega}_N) < \infty 
  \end{equation}
  holds. 
\end{definition} 

\begin{remark}
  Note that the inequalities \eqref{eq:stab-I} and \eqref{eq:stab-Q} can be generalized. 
  For instance, by applying the H\"older inequality, we get 
  \begin{align}
    | I[f] - I[\tilde{f}] | & \leq ||f-\tilde{f}||_{L^p} \cdot ||\omega||_{L^q}, \\ 
    | Q[f] - Q[\tilde{f}] | & \leq ||f-\tilde{f}||_{p} \cdot ||\vec{\omega}_N||_{q}
  \end{align}
  with $1\leq p,q \leq \infty$ and $\frac{1}{p} + \frac{1}{q} = 1$, assuming 
  $f-\tilde{f} \in C^0([a,b]) \cap L^p([a,b])$ and 
  ${\omega \in L^q([a,b])}$.
  Also Definition \ref{def:stability} could be adapted accordingly. 
  In this work, we only focus on the case of $p = \infty$ and $q = 1$. 
  This case allows us to simply bound the discrete norm $||f-\tilde{f}||_{p}$ by the continuous 
norm $||f-\tilde{f}||_{L^p}$; that is, $||f-\tilde{f}||_{p} \leq ||f-\tilde{f}||_{L^p}$. 
\end{remark}

Note that if all quadrature weights are nonnegative and the QRs have degree of exactness $0$, we 
have 
\begin{equation}
  \kappa(\vec{\omega}_N) 
    = \sum_{n=1}^N \omega_n 
    = I[1].
\end{equation}
If the quadrature weights have mixed signs, however, we get ${\kappa(\vec{\omega}_N) > I[1]}$.
This connection between stability of QRs and their weights being nonnegative has been 
utilized in \cite{huybrechs2009stable} and \cite[Chapter 4]{glaubitz2020shock} to construct stable QRs with high degrees of exactness on equidistant and scattered quadrature points. 
Yet, nonnegative weights are only reasonable if also the weight function $\omega$ in \eqref{eq:I} is nonnegative. 
Enforcing nonnegative quadrature weights for general weight functions cannot be considered as reasonable. 
Instead, it seems more convenient to fit the sign of a quadrature weight $\omega_n$ to the sign of the weight function 
$\omega$ at the corresponding quadrature point $x_n$. 
In this work, we refer to QRs with such weights as sign-consistent QRs.

\begin{definition}\label{def:sign-cons}
  Let $Q_N$ be an $N$-point QR with quadrature points $\{x_n\}_{n=1}^N$ and quadrature weights 
$\{ \omega_n \}_{n=1}^N$, which approximates the integral \eqref{eq:I} with weight function ${\omega:[a,b] \to \R}$. 
  We call the QR $Q_N$ \emph{sign-consistent} if 
  \begin{equation}
    \sign \left( \omega_n \right) = \sign \left( \omega(x_n) \right) 
  \end{equation}
  holds for all $n=1,\dots,N$.
\end{definition} 

In Definition \ref{def:sign-cons}, $\sign(\cdot)$ denotes the usual \emph{sign function} defined as 
\begin{equation}
  \sign: \R \to \R, \quad 
  \sign(x) :=
  \begin{cases}
    -1 & \text{if }  x < 0, \\
    \ \ 1 & \text{if } x \geq 0. 
  \end{cases}
\end{equation} 
Implicitly, we have already noted the following proposition. 

\begin{proposition}
   Let $\left( Q_N \right)_{N \in \N}$ be a sequence of $N$-point QRs approximating the integral 
\eqref{eq:I} with nonnegative weight function $\omega$. 
   If every $Q_N$ has degree of exactness $0$ and is sign-consistent, the sequence of QRs $\left( 
Q_N \right)_{N \in \N}$ is stable. 
\end{proposition}

\begin{proof} 
  Let $N \in \N$ and let $Q_N$ be the corresponding $N$-point QR with quadrature weights $\vec{\omega}_N$. 
  Since the weight function $\omega$ is nonnegative and the QR $Q_N$ has degree of exactness $0$ and is 
sign-consistent, all the quadrature weights are nonnegative, and we have 
  \begin{equation}\label{eq:proof-1}
    \kappa(\vec{\omega}_N) = Q_N[1] = I[1].
  \end{equation}
  The assertion follows from noting 
  \begin{equation}
    \sup_{N \in \N} \kappa(\vec{\omega}_N) = I[1] < \infty.
  \end{equation}
\end{proof}

Of course, there are stable (sequences of) QRs which are not sign-consistent, for instance, some  
composite Newton--Cotes rules. 
For nonnegative weight functions, sign-consistency can therefore be regarded as a stronger property than stability. 
For general weight functions, however, this connection breaks down as \eqref{eq:proof-1} is not satisfied anymore.
Still, a weaker connection can be shown and is stated in the following theorem. 

\begin{theorem}\label{thm:sign-stab-connection}
  Let $\left( Q_N \right)_{N \in \N}$ be a sequence of $N$-point QRs approximating the 
integral \eqref{eq:I} with weight function $\omega$. 
  Moreover, let every $Q_N$ be sign-consistent, and assume that 
  \begin{equation}\label{eq:thm-assumption}
    \lim_{N \to \infty} Q_N[\sign(\omega)] = I[\sign(\omega)]
  \end{equation}
  holds. 
  Then, we have 
  \begin{equation}
    \lim_{N \to \infty} \kappa(\vec{\omega}_N) = K_\omega.
  \end{equation}
\end{theorem}

\begin{proof}
  Since every $Q_N$ is sign-consistent and since $|r| = \sign(r)r$ holds for $r \in \R$, we have 
  \begin{equation}
  \begin{aligned}
    \kappa(\vec{\omega}_N) 
      = \sum_{n=1}^N \sign(\omega_n) \omega_n 
      = \sum_{n=1}^N \sign(\omega(x_n)) \omega_n 
      = Q_N[ \sign(\omega) ].
  \end{aligned}
  \end{equation}
  Next, assumption \eqref{eq:thm-assumption} provides us with  
  \begin{equation}
    \lim_{N \to \infty} \kappa(\vec{\omega}_N) = I[ \sign(\omega) ].
  \end{equation}
  Finally, this yields 
  \begin{equation}
    \lim_{N \to \infty} \kappa(\vec{\omega}_N) 
      = \int_a^b \sign( \omega(x) ) \omega(x) \intd x 
      = K_\omega
  \end{equation}
  and therefore the assertion.
\end{proof}

Note that for a nonnegative weight function $\omega \geq 0$, \eqref{eq:thm-assumption} can be replaced with 
\begin{equation}
  \lim_{N \to \infty} Q_N[1] = I[1],
\end{equation}
which is a fairly weak assumption. 
Yet, wheter \eqref{eq:thm-assumption} is also satisfied for a more general weight function depends on the weight function itself (e.\,g., if only finitely many jump discontinuities occur in $\sign(\omega)$) and how well the QRs $Q_N$ can handle these discontinuities. 
Note that assumption \eqref{eq:thm-assumption} could be ensured 
by adding $Q_N[ \sign(\omega) ] = I[ \sign(\omega) ]$ to the exactness conditions 
\eqref{eq:exactness-cond}---assuming the resulting system of linear equations remains solvable. 
This will be investigated in future works. 
Here, we treat stable (sequences of) QRs and sign-consistent QRs as two separate classes. 

\begin{remark}
  It should be noted that while sign-consistency might be related to stability, at least in some cases, it is neither 
necessary nor sufficient for exactness. 
  For instance, interpolatory QRs for $\omega(x) = 1$ (and equidistant quadrature points), having degree of exactness 
$d=N-1$, are known to have negative quadrature weights for $N=9$ and $N \geq 11$; see 
\cite[Chapter 4]{glaubitz2020shock}. 
  On the other hand, to show that sign-consistency is not sufficient for exactness, we can consider a 
simple Riemann sum ${Q_N[f] = \sum_{n=1}^N \frac{\omega(x_n)}{N} f(x_n)}$. 
  Obviously, this QR is sign-consistent, yet it does not even have degree of exactness $d=0$.  
\end{remark} 
\section{Stable QRs for general weight functions} 
\label{sec:LS-QRs}

In this section, we construct stable (sequences of) QRs for equidistant and scattered data with high 
degrees of exactness and for general weight functions.  
Let $d \in \N$, and let $\{(x_n,f(x_n))\}_{n=1}^N$ be a set of data with ${N > d}$. 
Further, let us denote the vector of quadrature points ${\left( x_1, \dots, x_N \right)^T \in [a,b]^N}$ by $\vec{x}_N$. 
We are interested in the construction of stable (sequences of) $N$-point QRs ${\left( Q_N \right)_{N \in 
\N}}$ with quadrature points ${\left( \vec{x}_N \right)_{N \in \N}}$ and degree of exactness $d$.

\subsection{Formulation as a LS problem} 

Let $\{\varphi_k\}_{k=0}^d$ be a basis of $\mathbb{P}_d([a,b])$, 
and let us choose $N > d$, i.\,e., a larger number of quadrature points than technically needed to 
construct an (interpolatory) QR with degree of exactness $d$. 
Then, the exactness conditions \eqref{eq:exactness-cond} become an underdetermined system of linear equations 
\begin{equation}\label{eq:linear-system}
  A \vec{\omega}_N = \vec{m} 
\end{equation} 
for the weights $\vec{\omega}_N \in \R^N$, where the matrix
\begin{equation}\label{eq:A}
  A = \left( \varphi_k(x_n) \right)_{k=0,n=1}^{d,N} \in \R^{(d+1)\times N}
\end{equation}
contains the function values of the basis elements at the quadrature points and the vector 
\begin{equation}\label{eq:m}
  \vec{m} = \left( I[\varphi_k] \right)_{k=0}^d \in \R^{d+1}
\end{equation}
contains the \emph{moments} of the basis elements. 
The underdetermined system of linear equations \eqref{eq:linear-system} induces an ($N-d-1$)-dimensional affine 
subspace of solutions \cite{golub2012matrix}, which we denote by 
\begin{equation}
  \Omega := \left\{ \vec{\omega}_N \big| \ A \vec{\omega}_N = \vec{m} \right\}.
\end{equation}
Note that for every $\vec{\omega}_N \in \Omega$, the resulting $N$-point QR has degree of exactness $d$. 
From $\Omega$ we want to determine the \emph{LS solution} $\vec{\omega}_N^{\mathrm{LS}}$ which minimizes the 
Euclidean norm
\begin{equation}
  \norm{\vec{\omega}_N}_2 := \left( \sum_{n=1}^N \left| \omega_n \right|^2 \right)^{\frac{1}{2}}, 
\end{equation}
i.\,e., we want to determine $\vec{\omega}_N^{\mathrm{LS}}$ such that
\begin{equation}
  \vec{\omega}_N^{\mathrm{LS}} = \argmin_{ \vec{\omega}_N \in \Omega } \ \norm{\vec{\omega}_N}_2 
\end{equation} 
is satisfied. 
The LS solution $\vec{\omega}_N^{\mathrm{LS}}$ exists and is unique \cite{golub2012matrix}. 
At least formally, $\vec{\omega}_N^{\mathrm{LS}}$ can be obtained by  
\begin{equation}\label{eq:ls-solution}
  \vec{\omega}_N^{\mathrm{LS}} = A^T \vec{u}, 
\end{equation} 
where $\vec{u}$ is the unique solution of the normal equation 
\begin{equation}\label{eq:normal-eq}
  A A^T \vec{u} = \vec{m}.
\end{equation}
Yet, the above representation should not be solved numerically, since the normal equation \eqref{eq:normal-eq} tends to 
be ill-conditioned. 
The real advantage of this approach is revealed when incorporating the concept of discrete orthogonal polynomials (DOPs).

\subsection{DOPs}

Let $\mathbb{P}_d([a,b])$ be the space of real polynomials on $[a,b]$ with degree at most $d$. 
If $N > d$, we can define an inner product and norm associated with the vector of 
(quadrature) points $\vec{x}_N$ for any pair ${f,g \in \mathbb{P}_d([a,b])}$ as 
\begin{equation}\label{eq:inner-prod}
  \scp{f}{g}_{\vec{x}_N} := \sum_{n=1}^N f(x_n) g(x_n), \quad 
  \norm{f}_{\vec{x}_N}^2 = \scp{f}{f}_{\vec{x}_N};
\end{equation}
see \cite{gautschi2004orthogonal}. 
Then, a basis $\{ \varphi_k \}_{k=0}^d$ of $\mathbb{P}_d([a,b])$ is called a basis of \text{DOPs} if 
\begin{equation}
  \scp{\varphi_k}{\varphi_l}_{\vec{x}_N} = \delta_{kl} :=
  \begin{cases}
    1 & \text{if } k = l, \\ 
    0 & \text{if } k \neq l
  \end{cases} 
\end{equation}
holds for all $k,l=0,\dots,d$. 
Using equidistant points $\vec{x}_N$ in $[a,b]$ with $x_1=a$ and $x_N=b$ results in the normalized discrete 
Chebyshev polynomials \cite[Chapter 1.5.2]{gautschi2004orthogonal}. 
In the Appendix, we have collected some of their properties which come in useful in \S \ref{sub:stability}. 
Using scattered points or a weighted inner product, however, often no explicit formula is known for bases of DOPs. 
In this case, they have to be constructed numerically. 
Available construction strategies are the Stieltjes procedure \cite[Chapter 2.2.3.1]{gautschi2004orthogonal}, the 
(classical) Gram--Schmidt process, and the modified Gram--Schmidt process \cite{trefethen1997numerical}. 
As do many other works \cite{gautschi1997numerical,gelb2008discrete,glaubitz2020shock}, we recommend constructing bases of DOPs using the modified Gram--Schmidt process.

\subsection{Characterizing the LS solution} 

When formulating the underdetermined system of linear equations \eqref{eq:linear-system} w.\,r.\,t.\ a basis of 
DOPs ${\{ \varphi_k \}_{k=0}^d}$, we get 
\begin{equation}
  A A^T = \left( \scp{\varphi_k}{\varphi_l}_{\vec{x}_N} \right)_{k,l=0}^d = I,
\end{equation}
and the normal equation \eqref{eq:normal-eq} reduces to  
\begin{equation}
  \vec{u} = \vec{m}. 
\end{equation}
In this case, the LS solution $\vec{\omega}_N^{\mathrm{LS}}$ is simply given by 
\begin{equation}
  \vec{\omega}_N^{\mathrm{LS}} = A^T \vec{m}. 
\end{equation}
Thus, we have 
\begin{equation}\label{eq:LS-QR-weights}
    \omega_n^{\mathrm{LS}} 
      = \sum_{k=0}^d \varphi_k(x_n) I[\varphi_k] 
      = \sum_{k=0}^d \varphi_k(x_n) \left( \int_a^b \varphi_k(x) \omega(x) \d x \right)
\end{equation}
for $n=1,\dots,N$.
Note that this formula is valid for any set of quadrature points as long as ${N > d}$. 
Finally, we define the $N$-points \emph{LS-QR} with quadrature points 
$\{x_n\}_{n=1}^N$ and with degree of exactness $d$ as 
\begin{equation}\label{eq:LS-QR}
  Q_{d,N}^{\mathrm{LS}}[f] := \sum_{n=1}^N \omega_n^{\mathrm{LS}} f(x_n),
\end{equation}
where the quadrature weights $\omega_n^{\mathrm{LS}}$ are given by \eqref{eq:LS-QR-weights}.

\begin{remark}\label{rem:comp-of-moments}
  It is a general problem that the construction of a QR requires the computation of moments. 
  Here, defining the quadrature weights of the LS-QR \eqref{eq:LS-QR} by \eqref{eq:LS-QR-weights}, we are in need of computing the moments 
  \begin{equation}\label{eq:compute-moments}
    I[\varphi_k] = \int_a^b \varphi_k(x) \omega(x) \d x
  \end{equation}
  for general bases of DOPs $\{ \varphi_k \}_{k=0}^d$.
  Depending on the weight function $\omega$ and the basis $\{ \varphi_k\}_{k=0}^d$, the exact evaluation of 
\eqref{eq:compute-moments} might be impractical or even impossible. 
  Hence, the basic idea is to approximate the continuous linear functional $I$ by a discrete linear functional 
$L_J$, which is unrelated to the LS-QR being created from the moments.
  In our implementation, we have calculated the moments $I[\varphi_k]$ numerically by using a Gauss--Legendre (GL) QR   
  \begin{equation}\label{eq:comp-moments}
    I[\varphi_k] 
      \approx L_J[\varphi_k] 
      := \sum_{j=1}^J \omega_j^{\mathrm{GL}} \varphi_k(x^{\mathrm{GL}}_{j}) \omega(x^{\mathrm{GL}}_{j})
  \end{equation}
  on a large set of GL points $\{ x^{\mathrm{GL}}_j \}_{j=1}^J$ and quadrature weights 
  $\{ \omega_j^{\mathrm{GL}} \}_{j=1}^J$.
  In \S \ref{sec:tests}, $J=200$ has been used for all tests. 
  The above procedure is reminiscent of earlier constructions, e.\,g., \cite{gautschi1968construction}. 
  Moreover, it should be noted that the above selection of a GL-QR limits the procedure to smooth weight functions 
$\omega$. 
  For insufficiently smooth weight functions $\omega$, depending on the properties of $\omega$, other QRs might 
be more appropriate, possibly using a larger number of quadrature points $J$ then. 
  In some cases, it might also be possible to use certain recurrence relations or differential/difference 
equations of the polynomials $\varphi_k$ (see \cite{gautschi2004orthogonal}) to simplify the computation of the 
moments.
\end{remark}

\begin{remark}\label{rem:complexity}
  The following list addresses the computational complexity of determining the LS quadrature weights 
$\vec{\omega}_N^{\mathrm{LS}}$ given by \eqref{eq:LS-QR-weights} w.\,r.\,t.\ the parameters $d$, $N$, and $J$ (as 
in Remark \ref{rem:comp-of-moments}).
  \begin{itemize}
    \item \emph{Computation of DOPs}: 
      Consulting \eqref{eq:LS-QR-weights}, we need to compute the values of the $d+1$ DOPs $\varphi_k$, 
${k=0,\dots,d}$, at the $N$ quadrature points $x_n$, ${n=1,\dots,N}$. 
      This is done by the modified Gram--Schmidt process, which yields asymptotic costs of $\mathcal{O}(N(d+1)^2)$; see 
      \cite[Chapter 5.2.8]{golub2012matrix}. 
    \item \emph{Computation of moments}: 
      The computation of the $d+1$ moments $m_k$, ${k=0,\dots,d}$, is described in Remark \ref{rem:comp-of-moments} 
(also see \eqref{eq:comp-moments}) and requires the values of the DOPs at $J$ points. 
      As described above, this yields asymptotic costs of $\mathcal{O}(J(d+1)^2)$. 
      Next, the summation in \eqref{eq:comp-moments} has to be performed for all $d+1$ moments and results in asymptotic 
costs of $\mathcal{O}(J(d+1))$. 
      Overall, the moments are computed with asymptotic costs of $\mathcal{O}(J(d+1)^2)$. 
    \item \emph{Computation of LS weights}: 
      Finally, the $N$ LS quadrature weights $\omega_n^{\mathrm{LS}}$, ${n=1,\dots,N}$, are computed in 
\eqref{eq:LS-QR-weights} by a simple summation of the $d+1$ products $\varphi_k(x_n) m_k$. 
      This is done with asymptotic costs of $\mathcal{O}(N(d+1))$. 
  \end{itemize}
  Altogether, the computational complexity of determining the LS quadrature weights is given by 
  \begin{equation}
    \underbrace{\mathcal{O}(N(d+1)^2)}_{\text{DOPs}} + 
    \underbrace{\mathcal{O}(J(d+1)^2)}_{\text{moments}} + 
    \underbrace{\mathcal{O}(N(d+1))}_{\text{LS weights}} = 
    \mathcal{O}((N+J)d^2).
  \end{equation}
  Hence, the computational complexity is linear in $N$ and $J$ and quadratic in $d$.
\end{remark}

\subsection{Stability} 
\label{sub:stability}

In this section, we investigate stability of (sequences of) LS-QRs  
$( Q_{d,N}^{\mathrm{LS}} )_{n \in \N}$. 
Our main result is Theorem \ref{thm:stability-eq}, which ensures stability for LS-QRs on equidistant points with any 
fixed degree of exactness. 

\begin{theorem}\label{thm:stability-eq}
  Let $d \in \N$ and let $Q_{d,N}^{\mathrm{LS}}$ be an $N$-point LS-QR with equidistant 
quadrature points on $[a,b]$ and with degree of exactness $d$. 
  Then, the sequence of LS-QRs ${( Q_{d,N}^{\mathrm{LS}} )_{N > d}}$ is stable, 
i.\,e.,
  \begin{equation}
    \sup_{N > d} \kappa(\vec{\omega}_N^{\mathrm{LS}}) < \infty
  \end{equation}
  holds, where $\vec{\omega}_N^{\mathrm{LS}}$ are the weights of $Q_{d,N}^{\mathrm{LS}}$.
\end{theorem}

\begin{proof}
  Let $d \in \N$ and $N > d$. 
  First, we note that 
  \begin{equation}
    \kappa(\vec{\omega}_N^{\mathrm{LS}}) 
      = \ \sum_{n=1}^N \left| \omega_n^{\mathrm{LS}} \right| 
      \overset{\eqref{eq:LS-QR-weights}}{\leq} \ \sum_{n=1}^N \sum_{k=0}^d \left| \varphi_k(x_n) \right| \left| I[\varphi_k] \right|.
  \end{equation}
  A "quick and dirty" estimate provides us with 
  \begin{align}
    \left| \varphi_k(x_n) \right| 
      & \leq \max_{x \in [a,b]} \left| \varphi_k(x) \right| =: \norm{\varphi_k}_{L^\infty}, \\ 
    \left| I[\varphi_k] \right| 
      & \leq \int_a^b \left| \varphi_k(x) \right| \left| \omega(x) \right| \d x 
      \leq \norm{\varphi_k}_{\infty} K_{\omega}; 
  \end{align}
  see \eqref{eq:K-omega}. 
  Thus, we have 
  \begin{equation}
    \kappa(\vec{\omega}_N^{\mathrm{LS}}) 
      \leq \sum_{n=1}^N \sum_{k=0}^d \norm{\varphi_k}_{\infty}^2 K_{\omega} 
      = N K_{\omega} \sum_{k=0}^d \norm{\varphi_k}_{\infty}^2.
  \end{equation}
  Since the quadrature points $\{x_n\}_{n=0}^N$ are equidistant (including the boundary points), the DOPs $\{ \varphi_k 
\}_{k=0}^d$ are given by the transformed and normalized discrete Chebyshev 
polynomials \eqref{eq:DOPs-eq}, as described in the Appendix. 
  Choosing 
  \begin{equation}
    N \geq N(d) := \frac{1}{2} [ (2d-1)^2 + 1 ],
  \end{equation} 
  condition \eqref{eq:cond-k-N} is satisfied for all 
$k=0,\dots,d$, and the DOPs $\varphi_k$ are bounded by 
  \begin{equation}\label{eq:crucial-inequality}
    \norm{\varphi_k}_\infty \leq \frac{1}{\sqrt{h_k}}; 
  \end{equation}
  see \eqref{eq:bound-inf}. 
  Further, this yields 
  \begin{equation}\label{eq:proof2} 
    \norm{\varphi_k}_\infty^2 
      \overset{\eqref{eq:h_k}}{=} \frac{(2k+1) (N-1)! (N-1)!}{(N+k)! (N-k-1)!} 
      = \ \frac{(2k+1) (N-1)! k!}{(N+k)!} \binom{N-1}{k}. 
  \end{equation}
  Next, we note the simple inequalities 
  \begin{equation}
    \binom{N-1}{k} \leq \frac{(N-1)^k}{k!} 
  \end{equation}
  and 
  \begin{equation} 
    \frac{(N-1)!}{(N+k)!} 
      = \frac{1}{N(N+1)\dots(N+k)} 
      \leq \frac{1}{N^{k+1}}. 
  \end{equation}
  Incorporating these inequalities into \eqref{eq:proof2}, we get 
  \begin{equation} 
    \norm{\varphi_k}_\infty^2
      \leq \frac{(2k+1) k! (N-1)^k}{N^{k+1} k!} 
      < \frac{(2k+1)}{N},
  \end{equation}
  yielding 
  \begin{equation} 
    \kappa(\vec{\omega}_N^{\mathrm{LS}}) 
      \leq N K_{\omega} \sum_{k=0}^d \frac{(2k+1)}{N} 
      = K_{\omega} (d+1)^2 
  \end{equation} 
  for $N \geq N(d)$.
  Finally, the claim follows from 
  \begin{equation}
    \kappa(\vec{\omega}_N^{\mathrm{LS}}) 
      \leq \max\left\{ \kappa(\vec{\omega}_{d+1}^{\mathrm{LS}}), \dots, \kappa(\vec{\omega}_{N(d)-1}^{\mathrm{LS}}), 
K_{\omega} (d+1)^2 
\right\}.
  \end{equation}
\end{proof}

\begin{remark}
  We note that in the above proof we have essentially bounded the stability measure 
$\kappa(\vec{\omega}_N^{\mathrm{LS}})$ by $K_{\omega} (d+1)^2$. 
  For increasing degree of exactness $d$, this constant becomes fairly large, and one might say that the LS-QRs are only 
\emph{pseudostable} for general weights. 
  In this case, the growth of errors is still bounded, yet only by a large factor. 
  In our numerical tests, however, we have observed the stability measure $\kappa(\vec{\omega}_N^{\mathrm{LS}})$ to 
converge to fairly small limits. 
  Moreover, note that in many cases $\kappa(\vec{\omega}_N^{\mathrm{LS}})$ converges to 
the continuous bound $K_\omega$ given by \eqref{eq:K-omega} for all degrees of exactness $d$; 
  see Theorem \ref{thm:sign-stab-connection}. 
  This indicates that the LS-QRs have even more desirable stability properties than we 
have proven in Theorem \ref{thm:stability-eq}.
\end{remark}

The proof of Theorem \ref{thm:stability-eq} heavily relies on the bound 
\begin{equation}
  \norm{\varphi_k}_\infty^2 \leq \frac{(2k+1)}{N}
\end{equation}
for the DOPs $\varphi_k$ with $N \geq N(k)$. 
Unfortunately, we have not been able to find similar results for DOPs on nonequidistant 
points 
in the literature. 
Yet, by assuming a similar bound, we can also prove the following theorem. 

\begin{theorem}\label{thm:stability-rand}
  Let $d \in \N$, and let $Q_{d,N}^{\mathrm{LS}}$ be an $N$-point LS-QR with distinct 
points $\vec{x}_N$ in $[a,b]$ and with degree of exactness $d$. 
  Further, let ${\{ \varphi_k \}_{k=0}^d}$ be the basis of DOPs corresponding to the inner 
product $\scp{\cdot}{\cdot}_{\vec{x}_N}$, and assume that there are constants ${C(d) > 0}$ and ${N(d) > d}$ such 
that 
  \begin{equation}\label{eq:assumption}
    \norm{\varphi_k}_\infty^2 \leq \frac{C(d)}{N} \quad \forall k=0,\dots,d \ \forall N \geq N(d) .
  \end{equation}
  Then, the sequence of LS-QRs ${( Q_{d,N}^{\mathrm{LS}}[f] )_{N > d}}$ is stable, i.\,e., 
  \begin{equation}
    \sup_{N > d} \kappa(\vec{\omega}_N^{\mathrm{LS}}) < \infty
  \end{equation}
  holds, where $\vec{\omega}_N^{\mathrm{LS}}$ are the quadrature weights of $Q_{d,N}^{\mathrm{LS}}$.
\end{theorem}

\begin{proof} 
  The proof is analogous to the one of Theorem \ref{thm:stability-eq}. 
  Let $d \in \N$ and $N > d$.
  Again, we have 
  \begin{equation}
     \kappa(\vec{\omega}_N^{\mathrm{LS}}) 
      \leq N K_{\omega} \sum_{k=0}^d \norm{\varphi_k}_{\infty}^2.
  \end{equation}
  Assuming there are constants $C(d) > 0$ and $N(d) > d$ such that \eqref{eq:assumption} holds yields 
  \begin{equation}
    \norm{\varphi_k}_\infty^2 \leq \frac{C(d)}{N}
  \end{equation}
  for all $k=0,\dots,d$ and $N \geq N(d)$. 
  Thus, we get 
  \begin{equation} 
    \kappa(\vec{\omega}_N^{\mathrm{LS}}) 
      \leq N K_{\omega} \sum_{k=0}^d \frac{C(d)}{N} 
      \leq K_{\omega} (d+1) C(d)
  \end{equation}
  for $N \geq N(d)$. 
  Once more, the claim follows from 
  \begin{equation}
    \kappa(\vec{\omega}_N^{\mathrm{LS}}) 
      \leq \max\left\{ \kappa(\vec{\omega}_{d+1}^{\mathrm{LS}}), \dots, \kappa(\vec{\omega}_{N(d)-1}^{\mathrm{LS}}), 
K_{\omega} (d+1) C(d) 
\right\}.
  \end{equation}
\end{proof}

\section{Sign-consistent QRs for general weight functions} 
\label{sec:NNLS-QRs} 

In this section, we construct sign-consistent QRs with high 
degrees of exactness for scattered data and general weight functions. 
The idea for the construction of stable QRs in \S \ref{sec:LS-QRs} has essentially been to ensure a fixed 
degree of exactness $d$ and to optimize the stability measure $\kappa(\vec{\omega}_N)$ then. 
The idea behind our construction of sign-consistent QRs is somewhat reversed. 
First, we ensure that the QR is sign-consistent, and only afterwards we "optimize" the degree of exactness. 
This approach has already been proposed by Huybrechs \cite{huybrechs2009stable} for nonnegative weight functions. 
Here, we extend it to general weight functions.

\subsection{The NNLS problem} 

The so-called \emph{NNLS} problem is a constrained LS problem, where the 
solution entries are not allowed to become negative. 
Following \cite{lawson1995solving}, the problem is to 
\begin{equation}\label{eq:NNLS-problem}
  \text{minimize} \ \norm{ B \vec{x} - \vec{c} }_2 
  \quad \text{subject to} \quad 
  x_n \geq 0, \ n=1,\dots,N,
\end{equation}
where $B \in \R^{M \times N}$, $\vec{x} \in \R^N$, and $\vec{c} \in \R^M$. 
Moreover, $\norm{\cdot}_2$ again denotes the Euclidean norm, this time in $\R^M$. 
The NNLS problem \eqref{eq:NNLS-problem} always has a solution \cite{lawson1995solving}. 
Note that in our setting, the matrix $B$ is underdetermined, which results in infinitely many solutions.  
In this case, we seek a solution with maximal sparsity, i.\,e., the solution vector has as many zero entries as possible.

\subsection{Formulation as an NNLS problem} 

When concerned with sign-consistent QRs with degree of exactness $d$ for integrals with general weight 
functions $\omega$, we want to 
\begin{equation}\label{eq:NNLS-problem2}
  \text{minimize} \ \norm{ A \vec{\omega} - \vec{m} }_2
\end{equation}
subject to mixed inequality constraints 
\begin{equation}\label{eq:mixed-const}
  \begin{cases}
    \omega_n \geq 0 & \text{if } \omega(x_n) \geq 0, \\ 
    \omega_n < 0 & \text{if } \omega(x_n) < 0,
  \end{cases} \quad
  n=1,\dots,N.
\end{equation} 
Again, the matrix $A$ and the vector $\vec{m}$ are given by \eqref{eq:A} and \eqref{eq:m} and contain the function 
values of the basis elements at the quadrature points and the corresponding moments, respectively. 
Nevertheless, we are able to formulate this problem as a usual NNLS problem \eqref{eq:NNLS-problem} by introducing an 
auxiliary sign-matrix 
\begin{equation}
  S = \operatorname{diag}\left( \sign\left( \omega(x_1) \right), \dots, \sign\left( 
\omega(x_N) \right) \right).
\end{equation}
Then, the LS problem \eqref{eq:NNLS-problem2} with mixed inequality constraints \eqref{eq:mixed-const} can 
be compactly rewritten as 
\begin{equation}\label{eq:NNLS-problem3}
  \text{minimize} \ \norm{ A \vec{\omega} - \vec{m} }_2 
  \quad \text{subject to} \quad 
  S \vec{\omega} \geq 0, 
\end{equation}
where $S \vec{\omega} \geq 0$ is understood in an elementwise sense. 
Finally, by going over to an auxiliary vector 
\begin{equation}
  \vec{u} := S \vec{\omega},
\end{equation}
problem \eqref{eq:NNLS-problem3} can be reformulated as a usual NNLS problem,
\begin{equation}\label{eq:NNLS-problem4}
  \text{minimize} \ \norm{ A S \vec{u} - \vec{m} }_2 
  \quad \text{subject to} \quad 
  u_n \geq 0, \ n=1,\dots,N,
\end{equation} 
since $\vec{\omega} = S^{-1} \vec{u} = S \vec{u}$. 
Let us denote the vector of quadrature weights which result from the NNLS problem \eqref{eq:NNLS-problem4} by 
$\vec{\omega}_N^{\mathrm{NNLS}}$. 
We refer to the corresponding QR 
\begin{equation}
  Q_{d,N}^{\mathrm{NNLS}}[f] := \sum_{n=1}^N \omega_n^{\mathrm{NNLS}} f(x_n)
\end{equation}
as the $N$-points \emph{NNLS-QR} with quadrature points 
$\{x_n\}_{n=1}^N$ and with \emph{approximate degree of exactness} $d$.
Note that if $0 < ||A \vec{\omega} - \vec{m}||_2 < \varepsilon$, the NNLS-QR is not exact for all polynomials up to 
degree $d$. 
Thus, we say that $Q_{d,N}^{\mathrm{NNLS}}$ has ``approximate degree of exactness $d$''. 
Yet, for $\varepsilon \approx 0$, we can expect the NNLS-QR to only have small errors. 
This is further discussed in \S \ref{sec:tests}.  
\section{Numerical results} 
\label{sec:tests}

In this section, we numerically investigate the proposed LS-QR and NNLS-QR. 
We do so for two different weight functions 
\begin{equation}\label{eq:weight_functions}
  \omega(x) = x \sqrt{1-x^3} 
  \quad \text{and} \quad 
  \omega(x) = \cos(20 \pi x)
\end{equation}
with mixed signs on $[-1,1]$ as well as for equidistant and scattered quadrature points. 
The scattered quadrature points are obtained by adding white Gaussian noise to the set of equidistant quadrature 
points $\{x_n\}_{n=1}^N$. 
Thus, the scattered quadrature points $\{\tilde{x}_n\}_{n=1}^N$ are given by 
\begin{equation}\label{eq:noise}
  \tilde{x}_1 = -1, \quad 
  \tilde{x}_N = 1, \quad
  \tilde{x}_n = x_n +  Z_n 
  \quad \text{with} \quad 
  Z_n \in \mathcal{N}\left(0;(4N)^{-2}\right),  
\end{equation}
for $n=2,\dots,N-1$, where the $Z_n$ are independent, identically distributed, and further assumed to not be correlated 
with the $x_n$.

\subsection{Implementation details} 
\label{sub:implementation}

The subsequent numerical tests have been computed in \textsc{Matlab} with double precision. 
This yields a machine precision of $2^{-52}$ or approximately $2.22\cdot10^{-16}$. 
Further, as described in Remark \ref{rem:complexity}, the LS quadrature weights given by \eqref{eq:LS-QR-weights} are 
computed as follows: 
\begin{enumerate}
  \item Matrix $A$ and therefore the nodal values $\varphi_k(x_n)$ are computed by applying the Gram--Schmidt process 
to an initial basis of Legendre polynomials. 
  \item The moments $I[\varphi_k]$ are computed by determining the nodal values of the DOPs $\varphi_k$ (again by the 
Gram--Schmidt process) and of the weight function $\omega$ at a set of $J$ GL points and utilizing the 
GL-QR applied to the integrand $\varphi_k \omega$ then. 
  Also see Remark \ref{rem:comp-of-moments}. 
  \item Finally, the LS quadrature weights are obtained by summing up the products $\varphi_k(x_n) I[\varphi_k]$ over 
$k$; see \eqref{eq:LS-QR-weights}. 
\end{enumerate}
The NNLS quadrature weights, on the other hand, are computed by the available \textsc{Matlab} implementation 
(\emph{lsqnonneg}) of the following NNLS algorithm \cite[Chapter 23]{lawson1995solving}. 
Also see \cite{huybrechs2009stable} or the \textsc{Matlab} documentation of \emph{lsqnonneg}.

\subsection{Stability} 
\label{sub:stability-tests}

We start by investigating stability of the proposed QRs.  
Figure \ref{fig:stab} illustrates the difference between the stability measure $\kappa(\vec{\omega}_N)$ of the LS-QR and the NNLS-QR with degree of exactness $d=10$ and the stability factor $K_{\omega}$, as defined in \eqref{eq:K-omega}, for an increasing number $N$ of equidistant as well as scattered quadrature points.

\begin{figure}[!htb]
  \centering
  \begin{subfigure}[b]{0.45\textwidth}
    \includegraphics[width=\textwidth]{%
      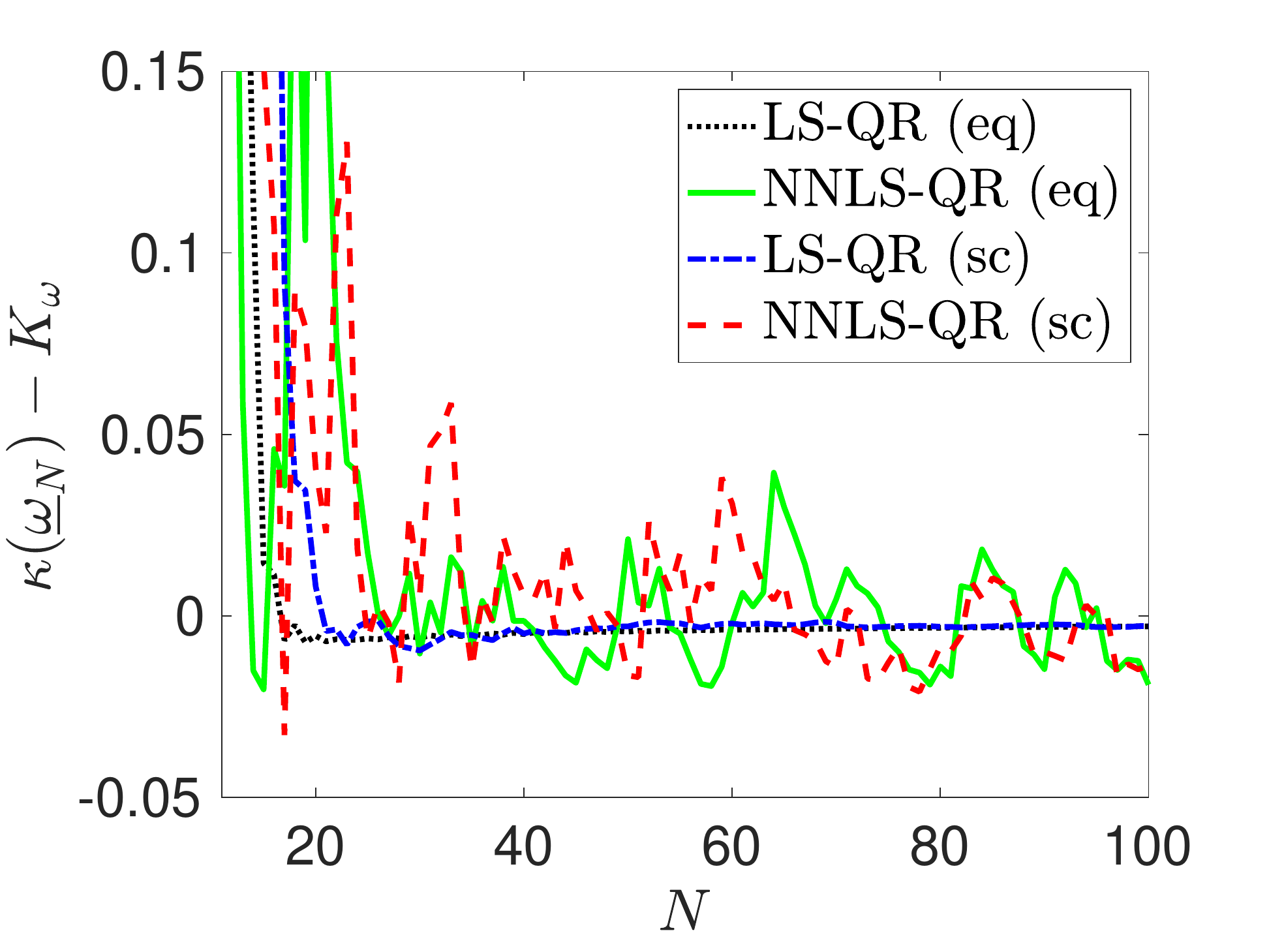} 
    \caption{${\omega(x) = x\sqrt{1-x^3}}$}
    \label{fig:stab-o1-d=10-kappa-K}
  \end{subfigure}%
  ~
  \begin{subfigure}[b]{0.45\textwidth}
    \includegraphics[width=\textwidth]{%
      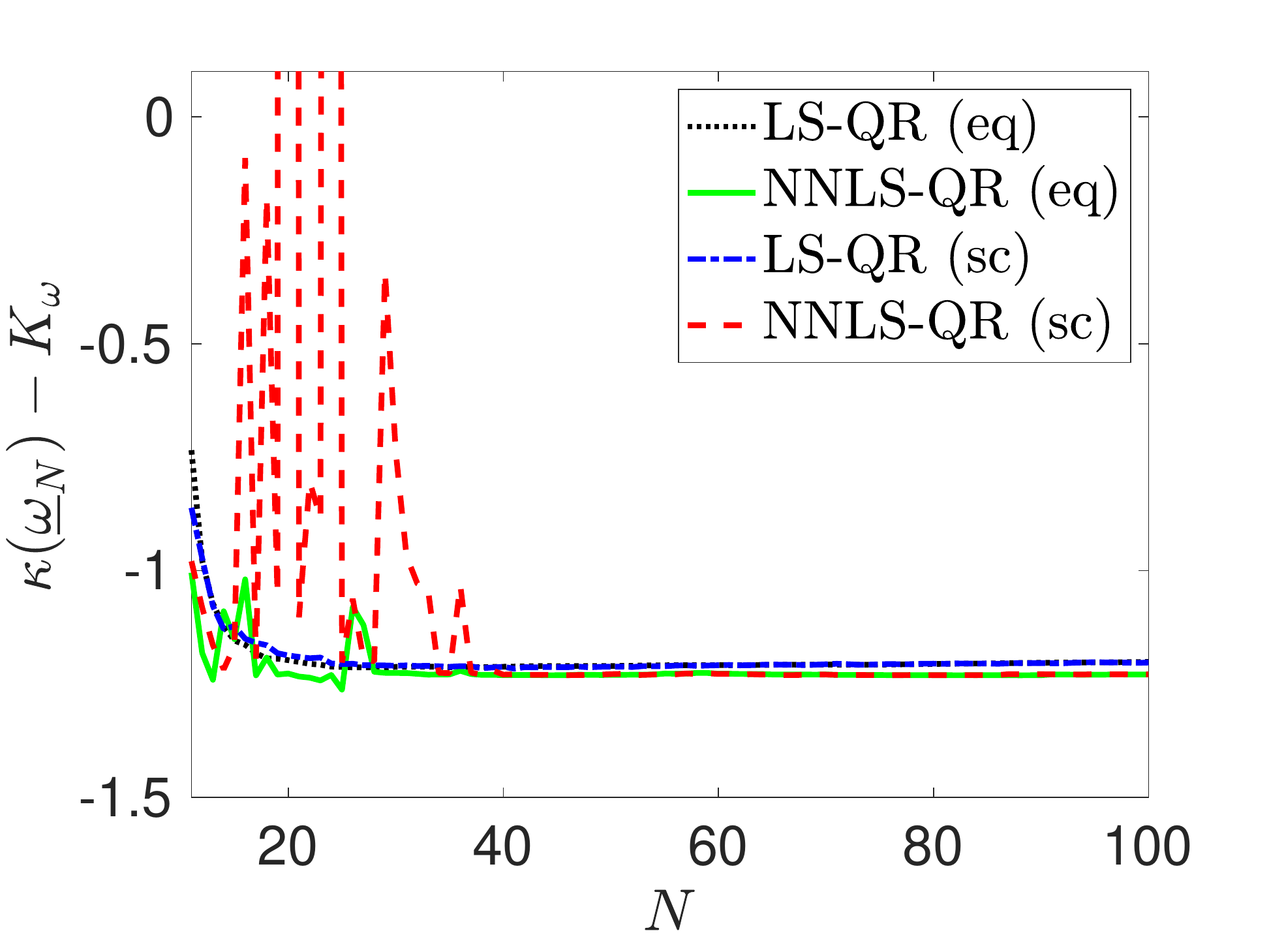}
    \caption{${\omega(x) = \cos(20 \pi x)}$}
    \label{fig:stab-o2-d=10-kappa-K}
  \end{subfigure}%
  \caption{Difference between the stability measure $\kappa(\vec{\omega}_N)$ and the stability factor $K_{\omega}$ for the LS-QR and NNLS-QR for $d=10$, equidistant (eq) as well as scattered (sc) quadrature points, and two different weight functions $\omega$.}
  \label{fig:stab}
\end{figure}

In accordance to Theorem \ref{thm:stability-eq} (and Theorem \ref{thm:stability-rand}), we note that the stability 
measure for the LS-QR is bounded in all cases for increasing $N$. 
A similar behavior can be observed for the NNLS-QR on equidistant as well as scattered quadrature points. 
Yet, the NNLS-QR shows a slightly oscillatory profile in some tests. 
Finally, Figure \ref{fig:stab} also demonstrates that in many cases the stability measure of the LS-QR and NNLS-QR is even lower than the one of the underlying integral $I$; see \eqref{eq:I}.

\subsection{Sign-consistency} 
\label{sub:sign-consistency}

Next, we investigate sign-consistency of the proposed LS-QR and NNLS-QR. 
We measure sign-consistency by the following \emph{sign-consistency measure}:
\begin{equation}
  S_{\omega}(\vec{\omega}_N) = \frac{1}{N} \sum_{n=1}^N \left| \sign(\omega_n) - \sign(\omega(x_n)) \right|.
\end{equation}
It is evident that an $N$-point QR $Q_N$ with quadrature weights $\vec{\omega}_N$ is sign-consistent if 
and only if ${S_{\omega}(\vec{\omega}_N) = 0}$ holds. 

\begin{figure}[!htb]
  \centering
  \begin{subfigure}[b]{0.45\textwidth}
    \includegraphics[width=\textwidth]{%
      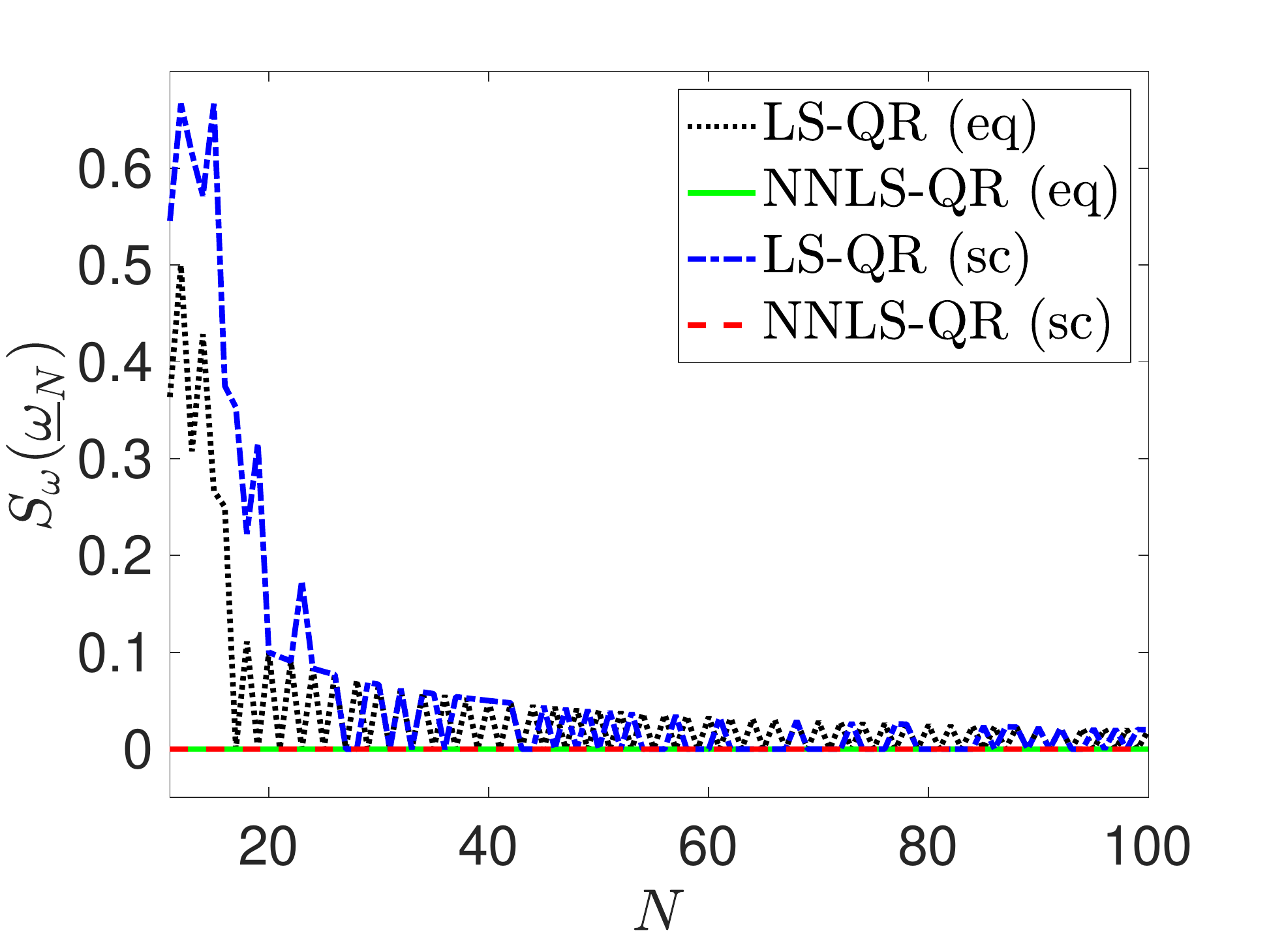}
    \caption{$\omega(x) = x\sqrt{1-x^3}$}
    \label{fig:sign-o1-d=10}
  \end{subfigure}%
  ~
  \begin{subfigure}[b]{0.45\textwidth}
    \includegraphics[width=\textwidth]{%
      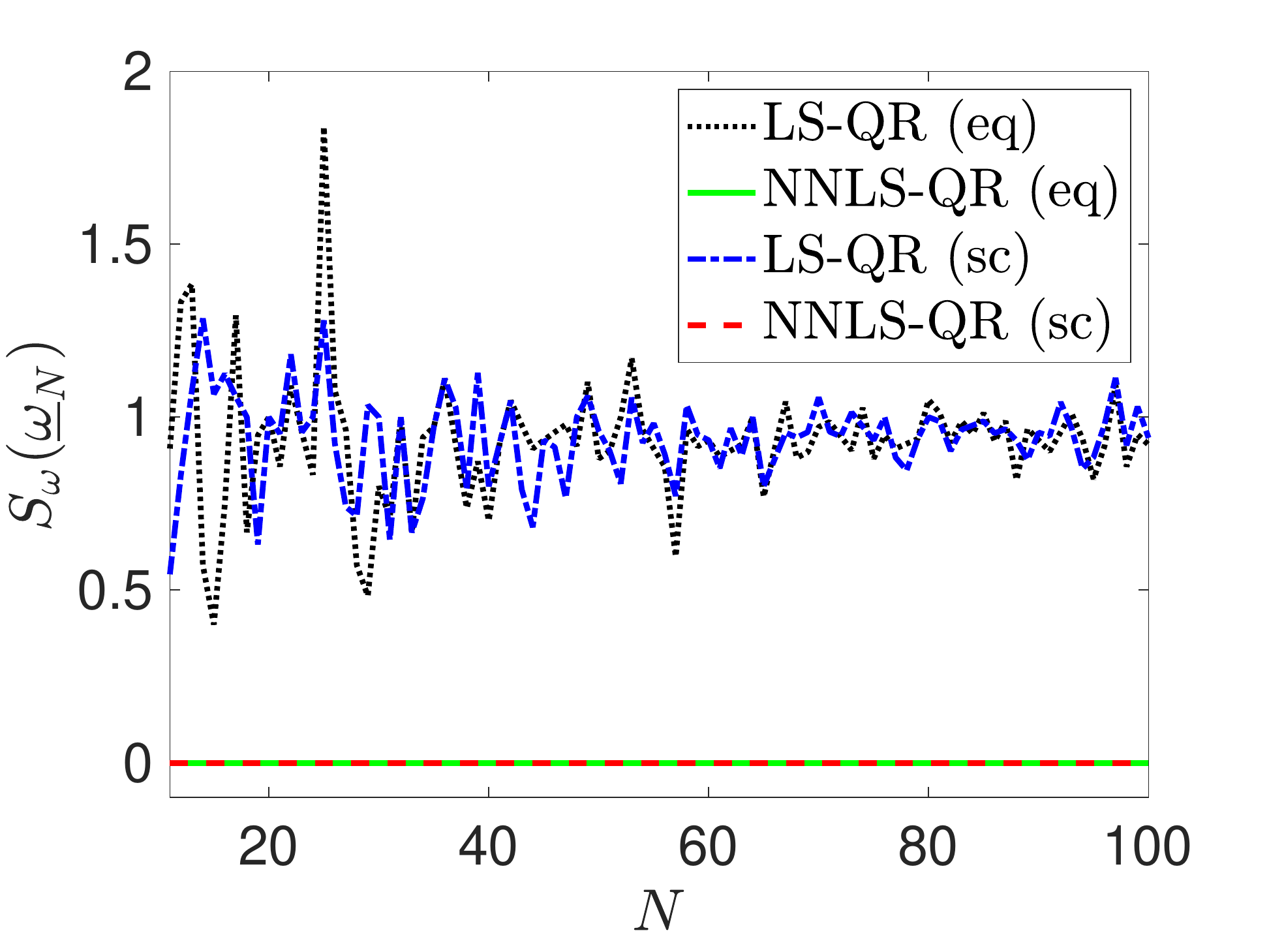}
    \caption{$\omega(x) = \cos(20 \pi x)$}
    \label{fig:sign-o2-d=10}
  \end{subfigure}%
  \caption{Sign-consistency measure $S_{\omega}(\vec{\omega}_N)$ for the LS-QR and NNLS-QR for $d=10$ and equidistant 
(eq) as well as scattered (sc) quadrature points.}
  \label{fig:sign}
\end{figure}

Figure \ref{fig:sign} illustrates the sign-consistency measure $S_{\omega}(\vec{\omega}_N)$ for the LS-QR and the 
NNLS-QR for the two weight 
functions \eqref{eq:weight_functions} with mixed signs, degree of exactness $d=10$, and equidistant as well 
as scattered quadrature points. 
We can observe that, in fact, the NNLS-QR always provides a sign-consistency 
measure of ${S_{\omega}(\vec{\omega}_N)=0}$. 
The LS-QR, on the other hand, is demonstrated to violate sign-consistency in many cases.
This indicates that the sign-consistent NNLS-QR might provide overall superior 
stability properties than the LS-QR.

\subsection{Exactness} 
\label{sub:exactness}

In this subsection, we investigate exactness of the LS-QR and the NNLS-QR.  
Figure \ref{fig:exact} displays the \emph{exactness measure} 
\begin{equation}
  \norm{ A \vec{\omega}_N - \vec{m} }_2
\end{equation}
of both QRs for the same set of parameters as before.

\begin{figure}[!htb]
  \centering
  \begin{subfigure}[b]{0.45\textwidth}
    \includegraphics[width=\textwidth]{%
      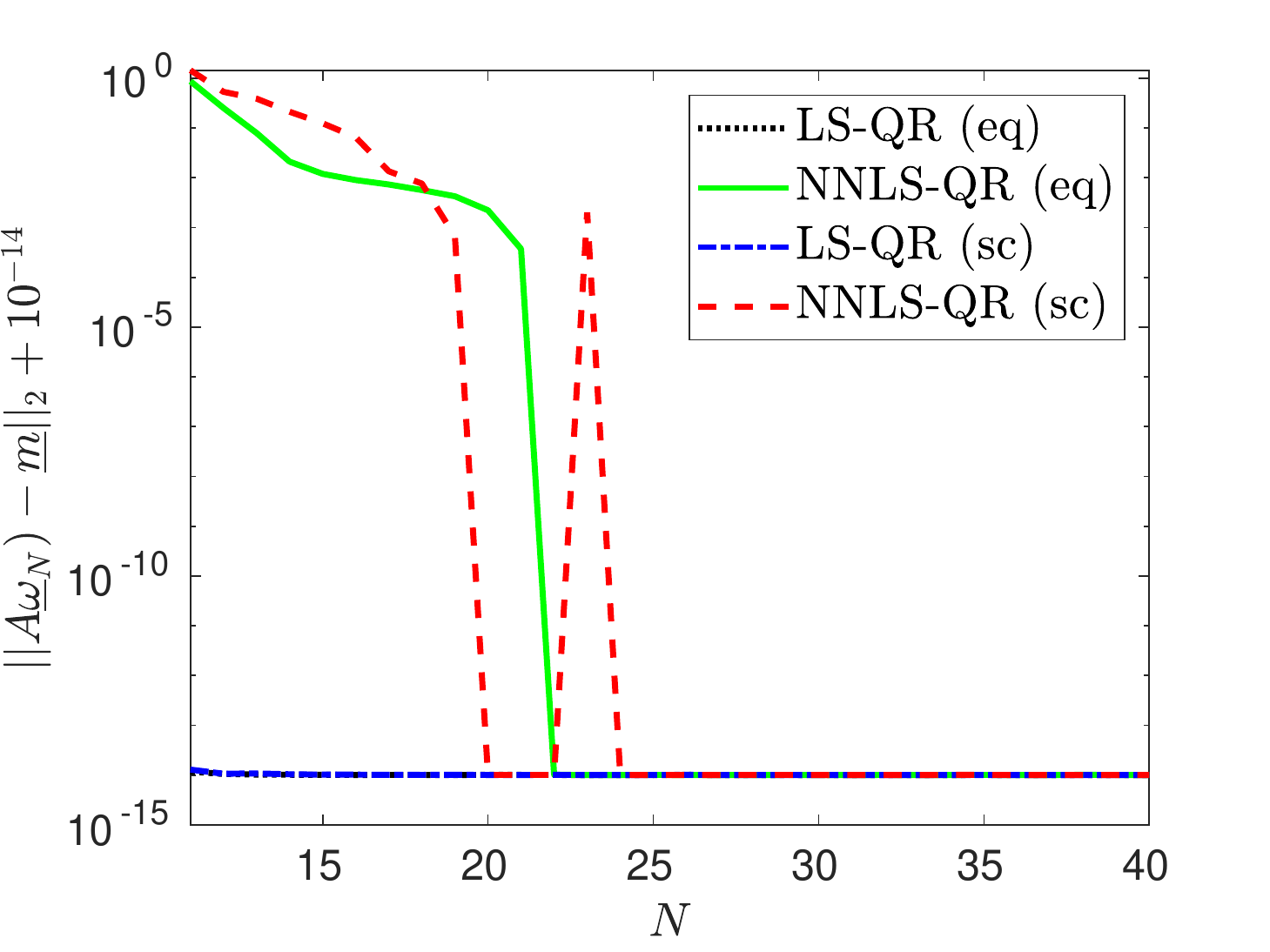}
    \caption{$\omega(x) = x\sqrt{1-x^3}$}
    \label{fig:exact-o1-d=10}
  \end{subfigure}%
  ~
  \begin{subfigure}[b]{0.45\textwidth}
    \includegraphics[width=\textwidth]{%
      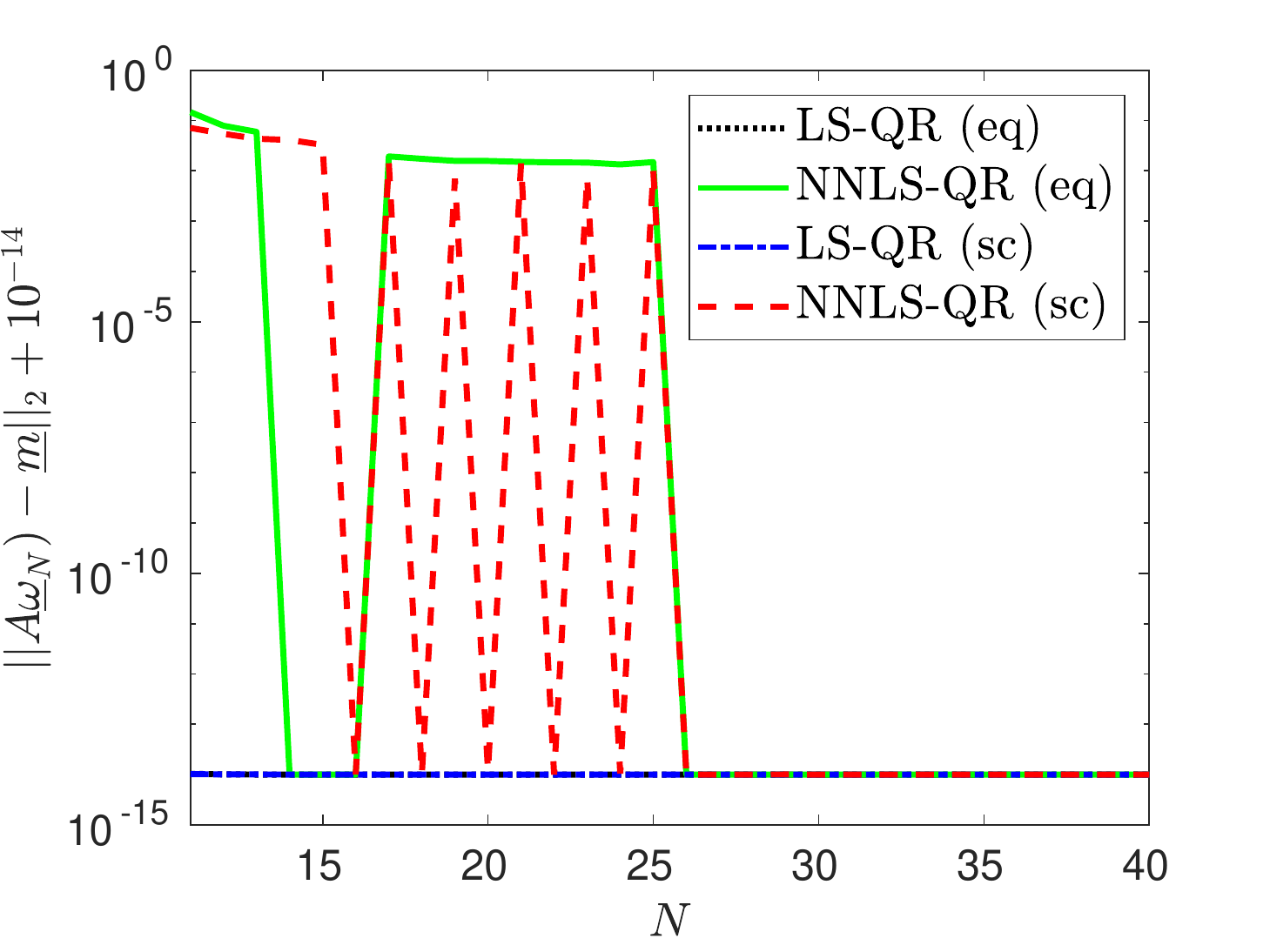}
    \caption{$\omega(x) = \cos(20 \pi x)$}
    \label{fig:exact-o2-d=10}
  \end{subfigure}%
  \caption{Exactness measure $||A \vec{\omega}^{\mathrm{NNLS}} - \vec{m}||_2$ for the LS-QR and NNLS-QR for $d=10$,  
equidistant (eq) as well as scattered (sc) quadrature points, and two different weight functions.}
  \label{fig:exact}
\end{figure}

Note that for an $N$-point QR $Q_N$ with quadrature weights $\vec{\omega}_N$ the minimal value for the 
exactness measure ${|| A \vec{\omega}_N - \vec{m} ||_2 = 0}$ holds if and only if $Q_N$ has degree of exactness 
$d$. 
In accordance to their construction as LS solutions of the linear system of exactness conditions 
\eqref{eq:linear-system}, the LS-QR provides a minimal value of ${|| A \vec{\omega}_N - 
\vec{m} ||_2 = 0}$ in all cases. 
The NNLS-QR, on the other hand, violates the degree of exactness and are 
demonstrated to have a nonzero exactness measure if insufficiently small numbers of quadrature points $N$ are used. 
Yet, for fixed degree of exactness $d$ and increasing $N$, also the NNLS-QR provides degree of exactness $d$.

\subsection{Accuracy for increasing $N$} 
\label{sub:conv-N}

We now come to investigate accuracy of the LS-QR and the NNLS-QR for a fixed 
degree of exactness $d=10$ and an increasing number of quadrature points $N$. 
We compare both QRs with a generalized composite trapezoidal rule applied to the unweighted Riemann 
integral with integrator $f\omega$. 
Further, we consider the two test functions $f(x) = |x|^3$ and $f(x) = e^x$. 
In the following, the errors 
\begin{equation}
  \left| Q_N[f] - I[f] \right|
\end{equation}
are illustrated.
Figure \ref{fig:conv-N-eq} shows the errors for equidistant quadrature points while Figure \ref{fig:conv-N-sc} provides 
the errors for  scattered quadrature points. 
These are once more constructed by adding white Gaussian noise to the set of equidistant quadrature 
points; see \eqref{eq:noise}.

\begin{figure}[!htb]
  \centering
  \begin{subfigure}[b]{0.4\textwidth}
    \includegraphics[width=\textwidth]{%
      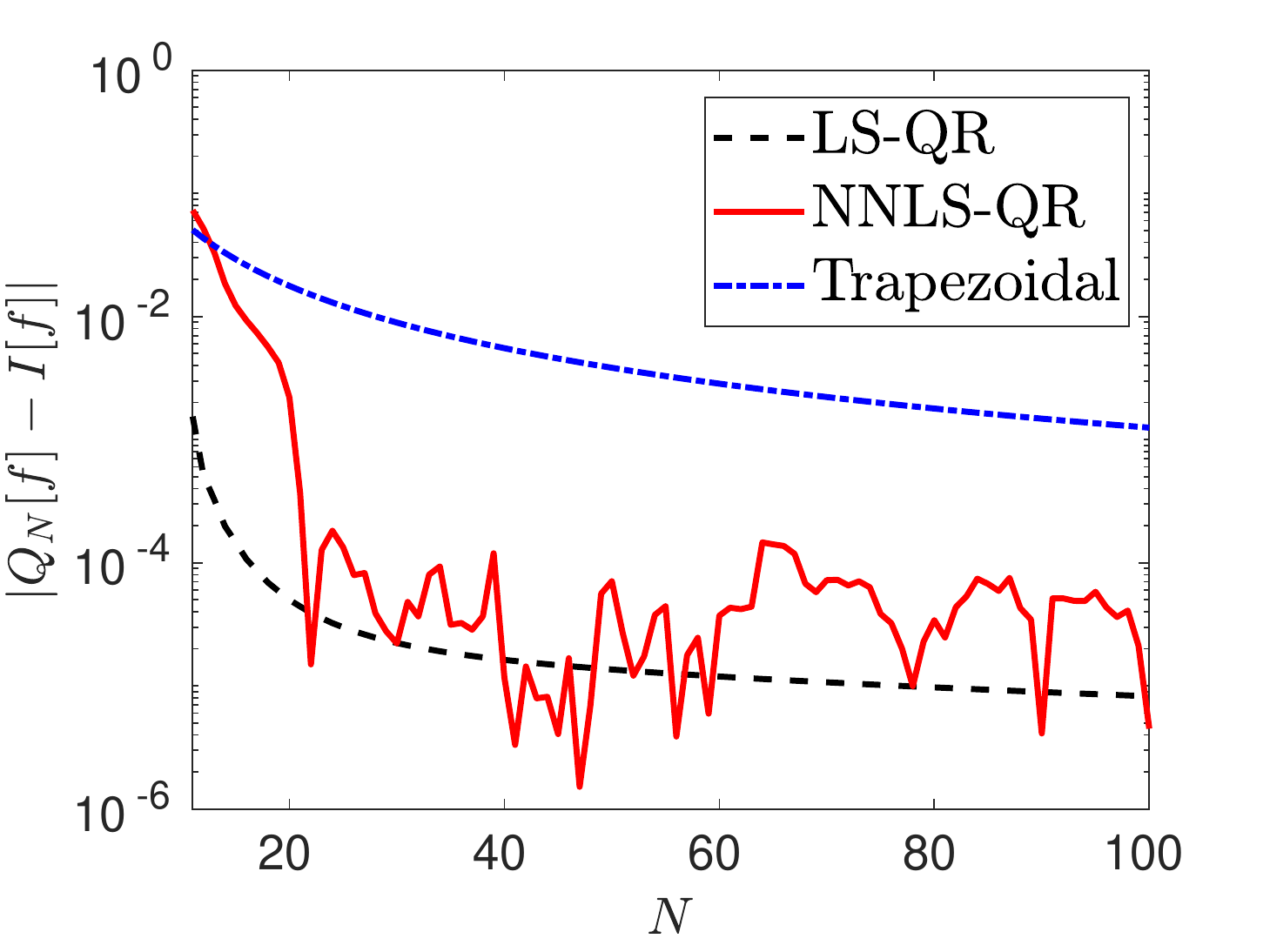}
    \caption{$f(x) = |x|^3$ \& $\omega(x) = x\sqrt{1-x^3}$}
    \label{fig:conv-N-eq-f1-o1-d=10}
  \end{subfigure}%
  ~
  \begin{subfigure}[b]{0.4\textwidth}
    \includegraphics[width=\textwidth]{%
      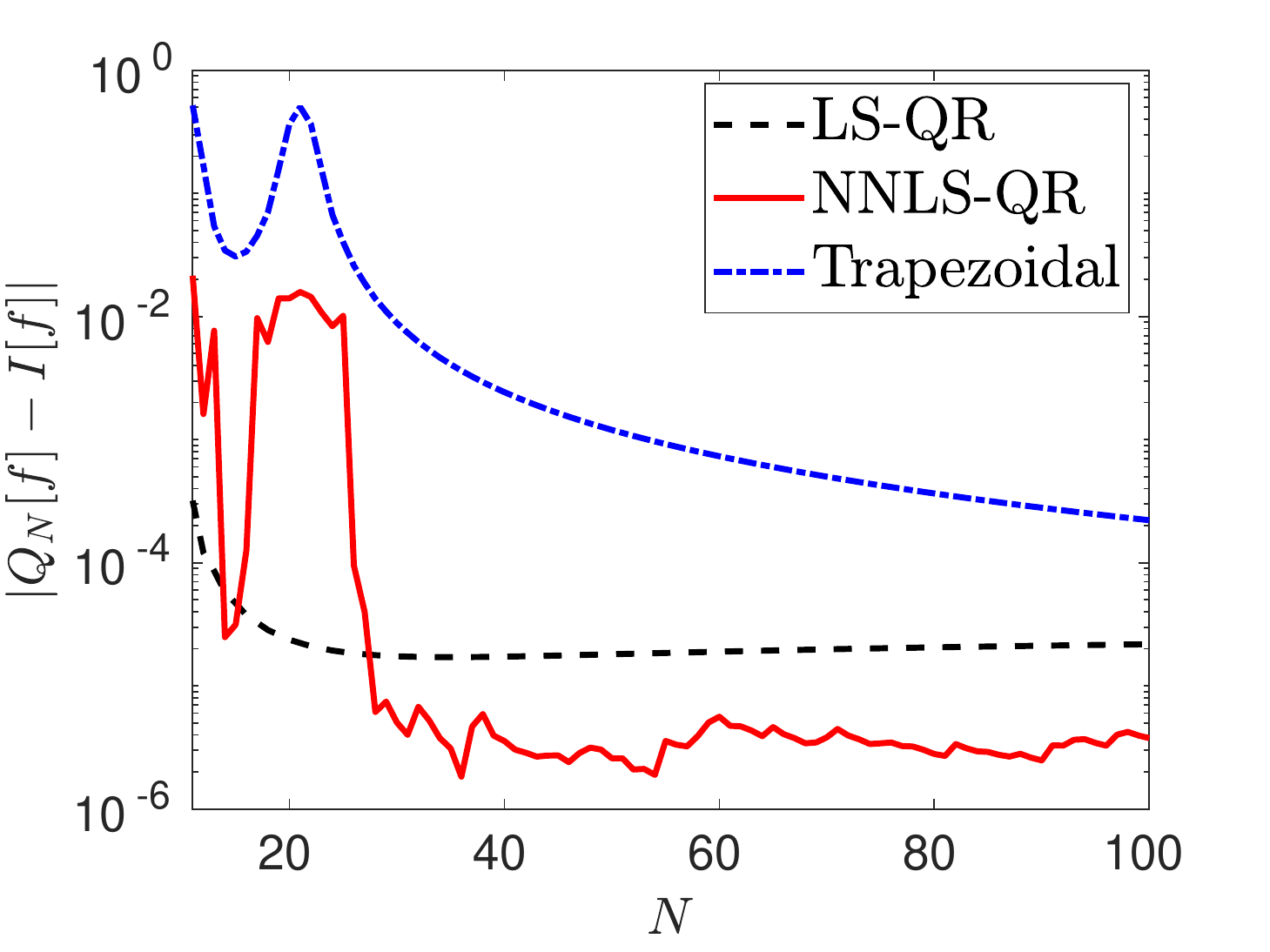}
    \caption{$f(x) = |x|^3$ \& $\omega(x) = \cos(20 \pi x)$}
    \label{fig:conv-N-eq-f1-o2-d=10}
  \end{subfigure}%
  \\
  \begin{subfigure}[b]{0.4\textwidth}
    \includegraphics[width=\textwidth]{%
      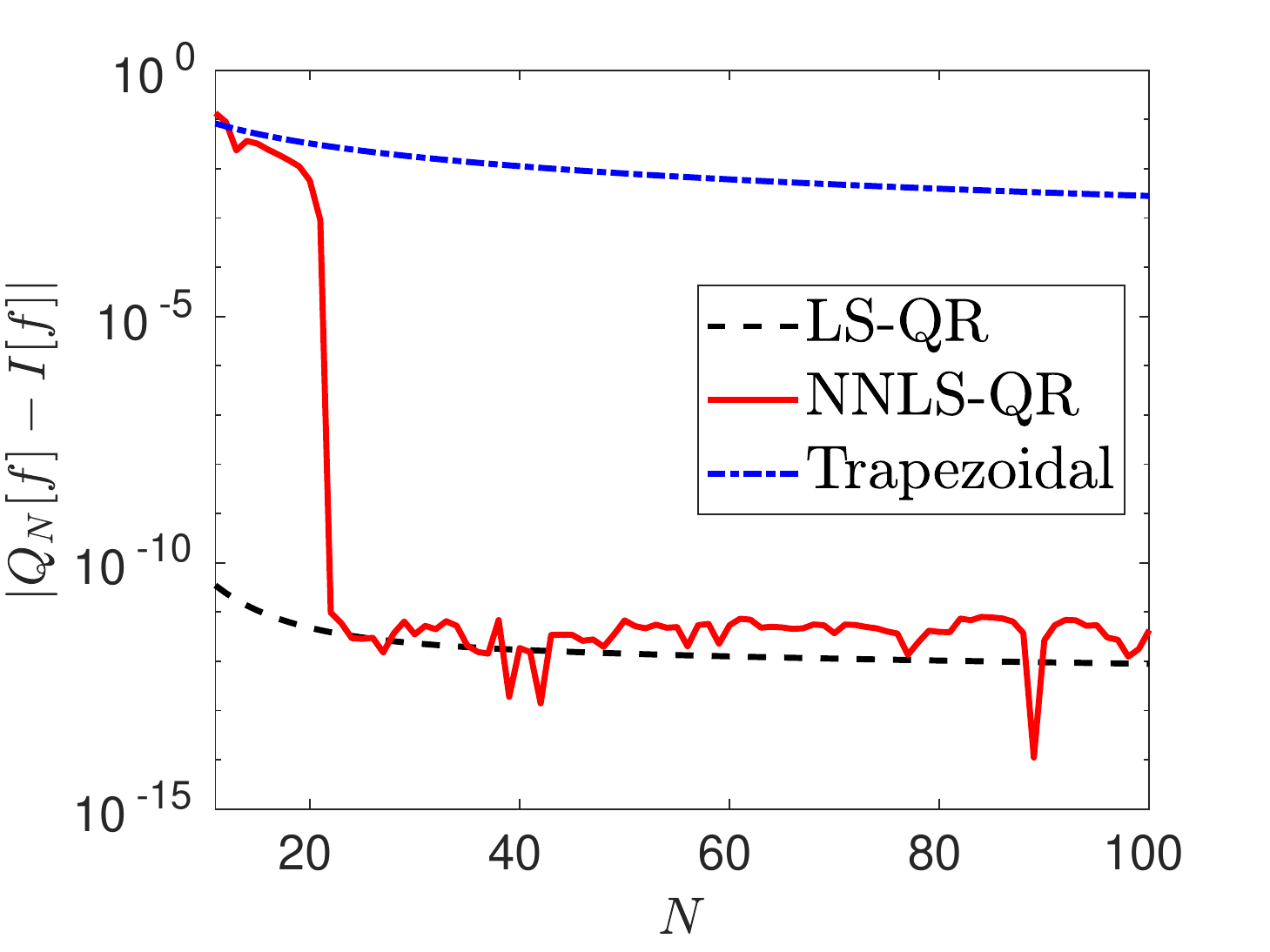}
    \caption{$f(x) = e^x$ \& $\omega(x) = x\sqrt{1-x^3}$}
    \label{fig:conv-N-eq-f2-o1-d=10}
  \end{subfigure}%
  ~
  \begin{subfigure}[b]{0.4\textwidth}
    \includegraphics[width=\textwidth]{%
      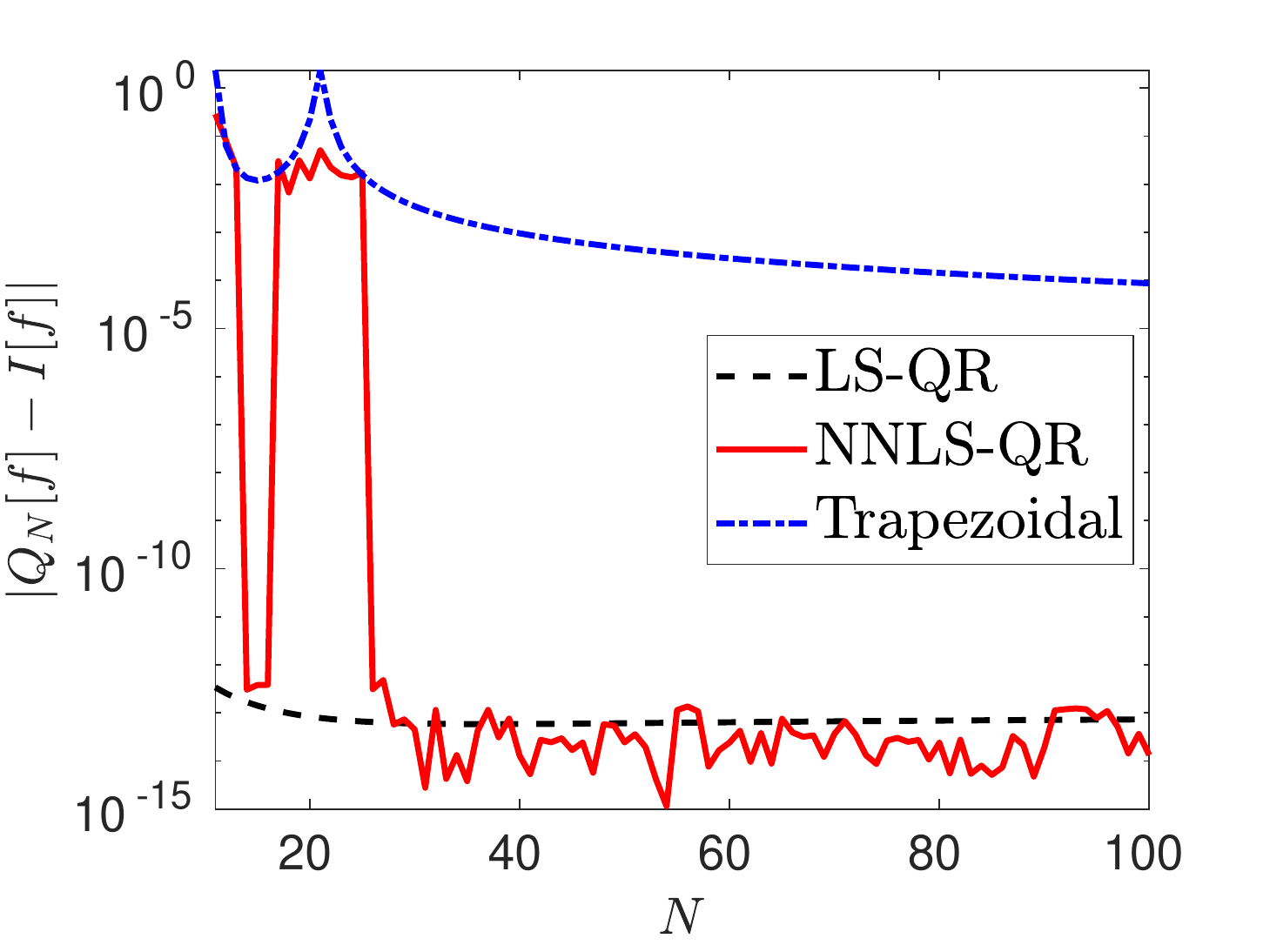}
    \caption{$f(x) = e^x$ \& $\omega(x) = \cos(20 \pi x)$}
    \label{fig:conv-N-eq-f2-o2-d=10}
  \end{subfigure}%
  \caption{Errors for the generalized composite trapezoidal rule, the LS-QR, and NNLS-QR with $d=10$ on equidistant 
points.}
  \label{fig:conv-N-eq}
\end{figure}

In Figure \ref{fig:conv-N-eq}, where equidistant quadrature points are investigated, the LS-QR and NNLS-QR are 
demonstrated to provide more accurate results than the generalized composite trapezoidal rule in nearly all cases. 
In fact, both QRs yield highly accurate results which are up to $10^{12}$ times more accurate than the ones 
obtained by the generalized composite trapezoidal rule at the same set of quadrature points. 
Yet, the NNLS-QR is observed to oscillate in its accuracy and to be less accurate for small 
numbers of quadrature points $N$. 
Note that for an insufficiently large number of quadrature points $N$, the exactness condition 
\eqref{eq:exactness-cond} is not satisfied by the NNLS-QR, and therefore not even constants might be handled accurately. 
For a sufficiently large number of quadrature points, however, the NNLS-QR provides similar accurate (sometimes even 
more accurate) results as the LS-QR. 

\begin{figure}[!htb]
  \centering
  \begin{subfigure}[b]{0.4\textwidth}
    \includegraphics[width=\textwidth]{%
      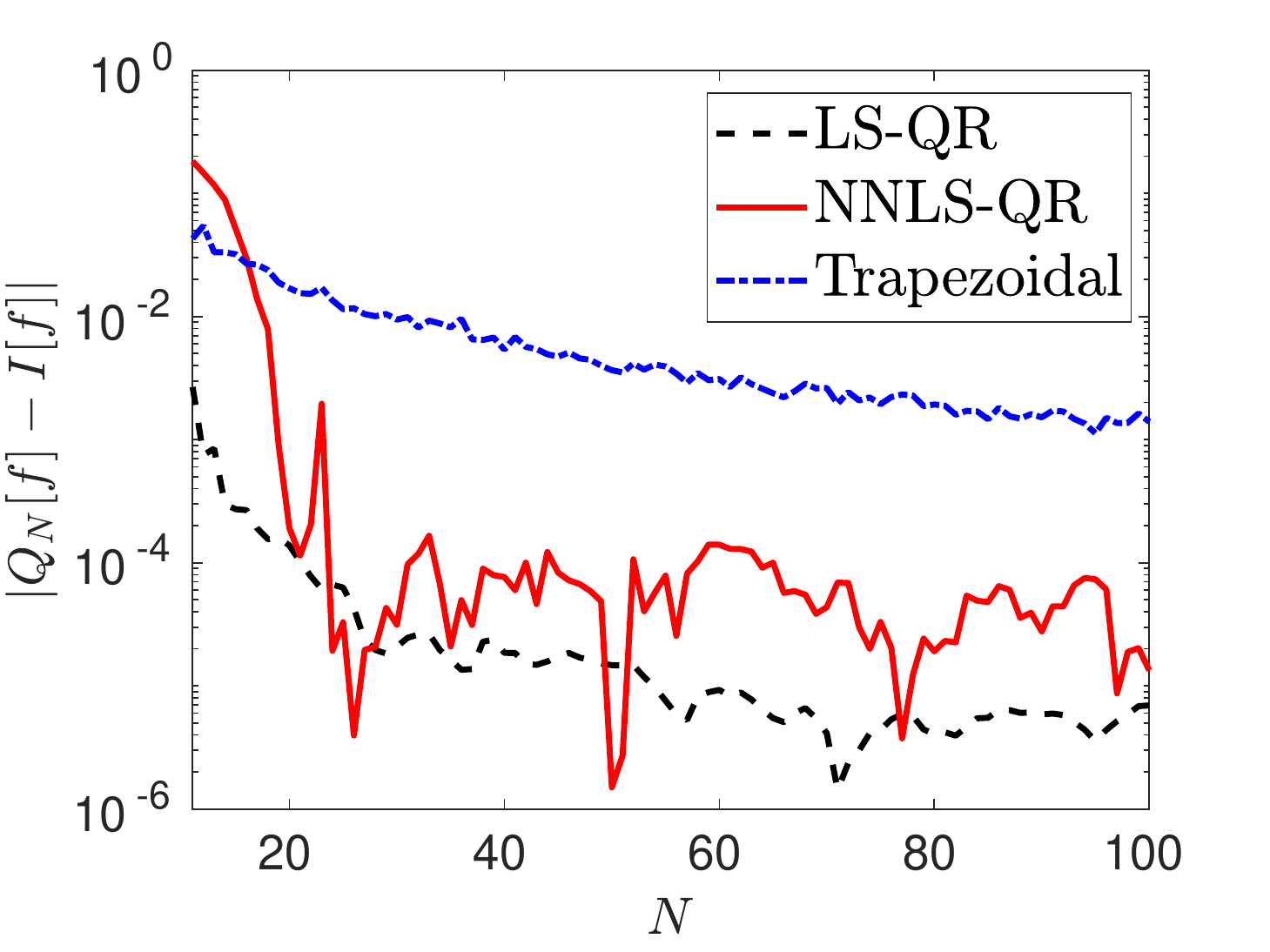}
    \caption{$f(x) = |x|^3$ \& $\omega(x) = x\sqrt{1-x^3}$}
    \label{fig:conv-N-sc-f1-o1-d=10}
  \end{subfigure}%
  ~
  \begin{subfigure}[b]{0.4\textwidth}
    \includegraphics[width=\textwidth]{%
      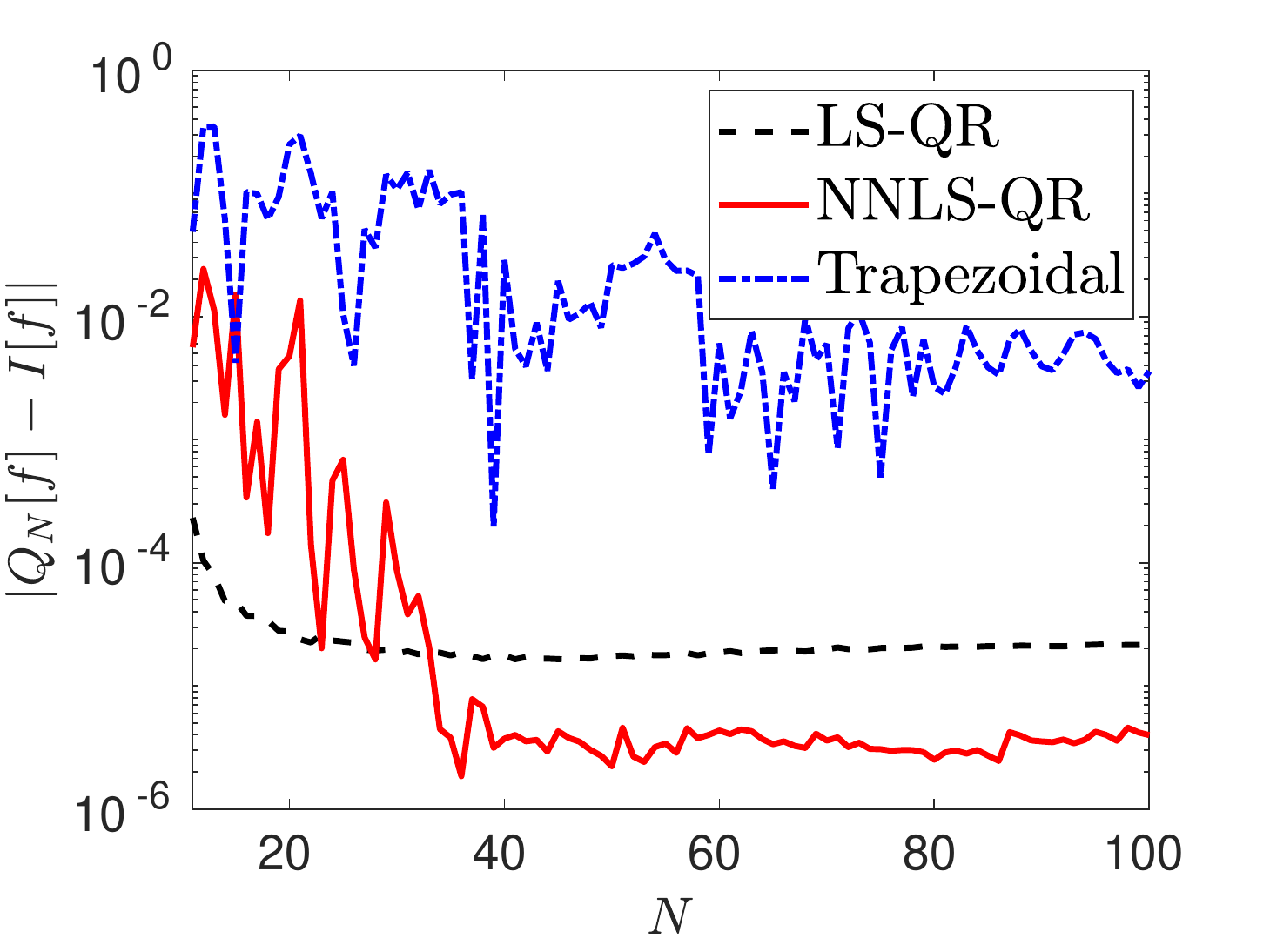}
    \caption{$f(x) = |x|^3$ \& $\omega(x) = \cos(20 \pi x)$}
    \label{fig:conv-N-sc-f1-o2-d=10}
  \end{subfigure}%
  \\
  \begin{subfigure}[b]{0.4\textwidth}
    \includegraphics[width=\textwidth]{%
      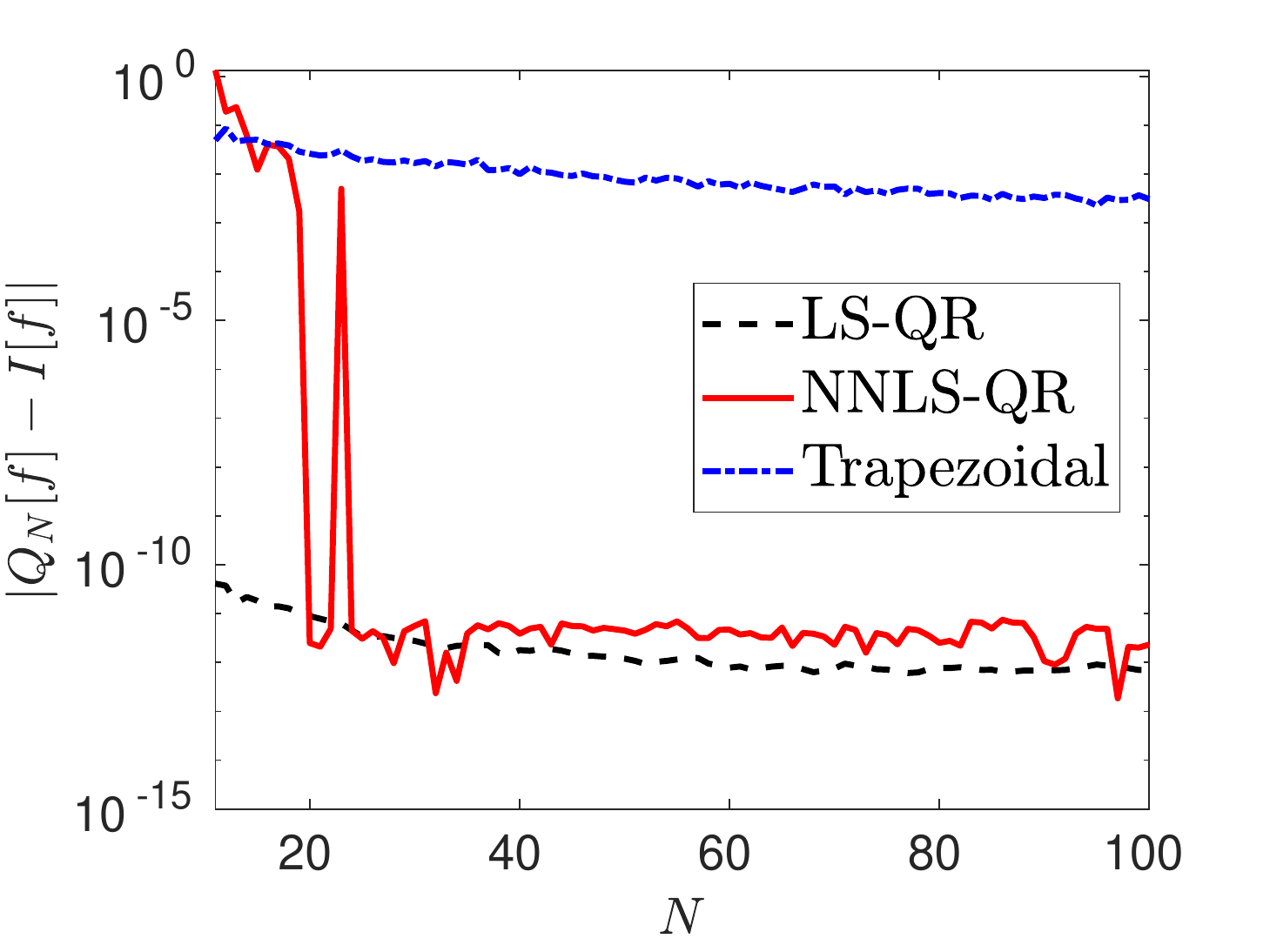}
    \caption{$f(x) = e^x$ \& $\omega(x) = x\sqrt{1-x^3}$}
    \label{fig:conv-N-sc-f2-o1-d=10}
  \end{subfigure}%
  ~
  \begin{subfigure}[b]{0.4\textwidth}
    \includegraphics[width=\textwidth]{%
      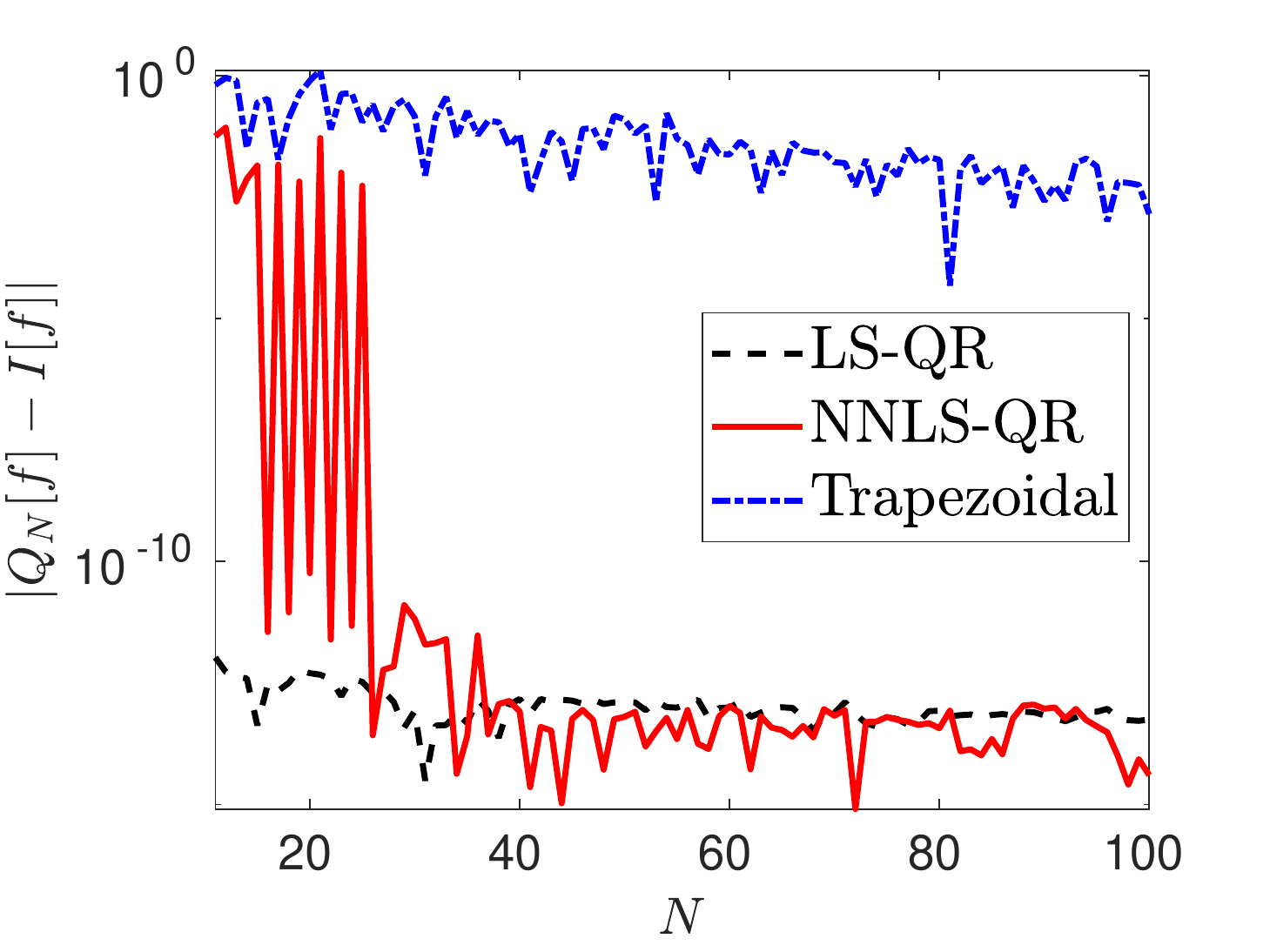}
    \caption{$f(x) = e^x$ \& $\omega(x) = \cos(20 \pi x)$}
    \label{fig:conv-N-sc-f2-o2-d=10}
  \end{subfigure}%
  \caption{Errors for the generalized composite trapezoidal rule, the LS-QR, 
and NNLS-QR with $d=10$ on scattered points.}
  \label{fig:conv-N-sc}
\end{figure}

Next, in Figure \ref{fig:conv-N-sc}, we consider scattered quadrature points. 
For scattered quadrature points the results and differences between the different QRs become less clear.  
Yet, the LS-QR and the NNLS-QR are observed to provide more accurate results than the generalized composite trapezoidal 
rule in most cases. 
Note that even on scattered quadrature points both QRs are up to $10^{12}$ times more accurate. 
Yet, the oscillations in the accuracy of the NNLS-QR become more pronounced on scattered quadrature points. 
Again, the NNLS-QR is demonstrated to provide highly accurate results only when a sufficiently large number of 
quadrature points $N$ is used.

\subsection{Accuracy for increasing $d$} 
\label{sub:conv-d} 

In the last section, we have considered accuracy for fixed degree of exactness $d$ and an increasing number of 
quadrature points $N$.  
In this section, we investigate accuracy for an increasing degree of exactness $d$ and adaptively choose 
${N = N(d) = \frac{1}{2}\left[ (2d-1)^2 + 1\right]}$. 
This choice is motivated by Theorem \ref{thm:stability-eq}. 
The most crucial step in the proof of Theorem \ref{thm:stability-eq} has been inequality \eqref{eq:crucial-inequality}, 
which holds for all ${k=0,\dots,d}$ if ${N \geq N(d) = \frac{1}{2}\left[ (2d-1)^2 + 1\right]}$ is chosen. 

\begin{figure}[!htb]
  \centering
  \begin{subfigure}[b]{0.4\textwidth}
    \includegraphics[width=\textwidth]{%
      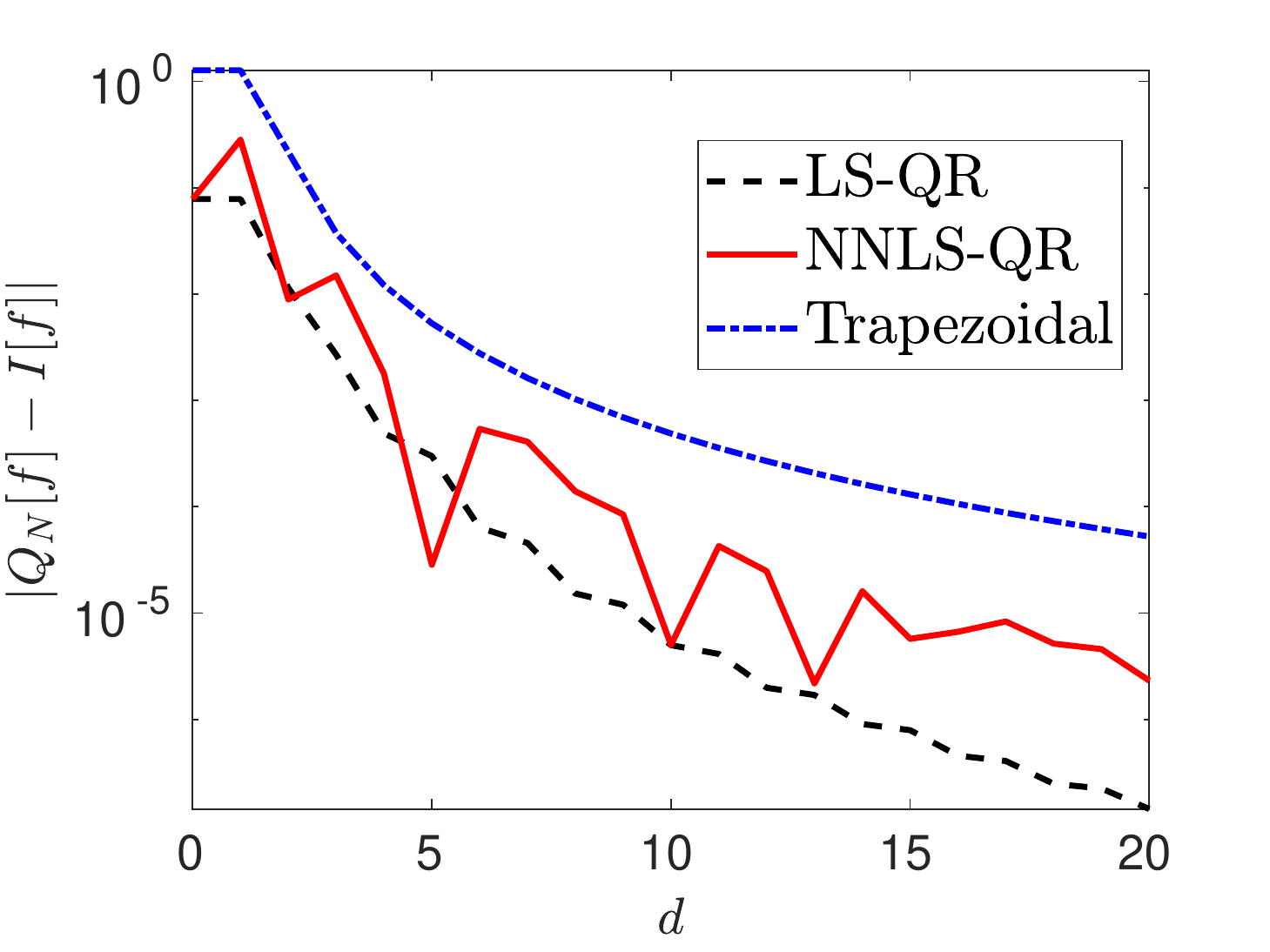}
    \caption{$f(x) = |x|^3 $ \& $\omega(x) = x\sqrt{1-x^3}$}
    \label{fig:accuracy-d-eq-o1-f1}
  \end{subfigure}%
  ~
  \begin{subfigure}[b]{0.4\textwidth}
    \includegraphics[width=\textwidth]{%
      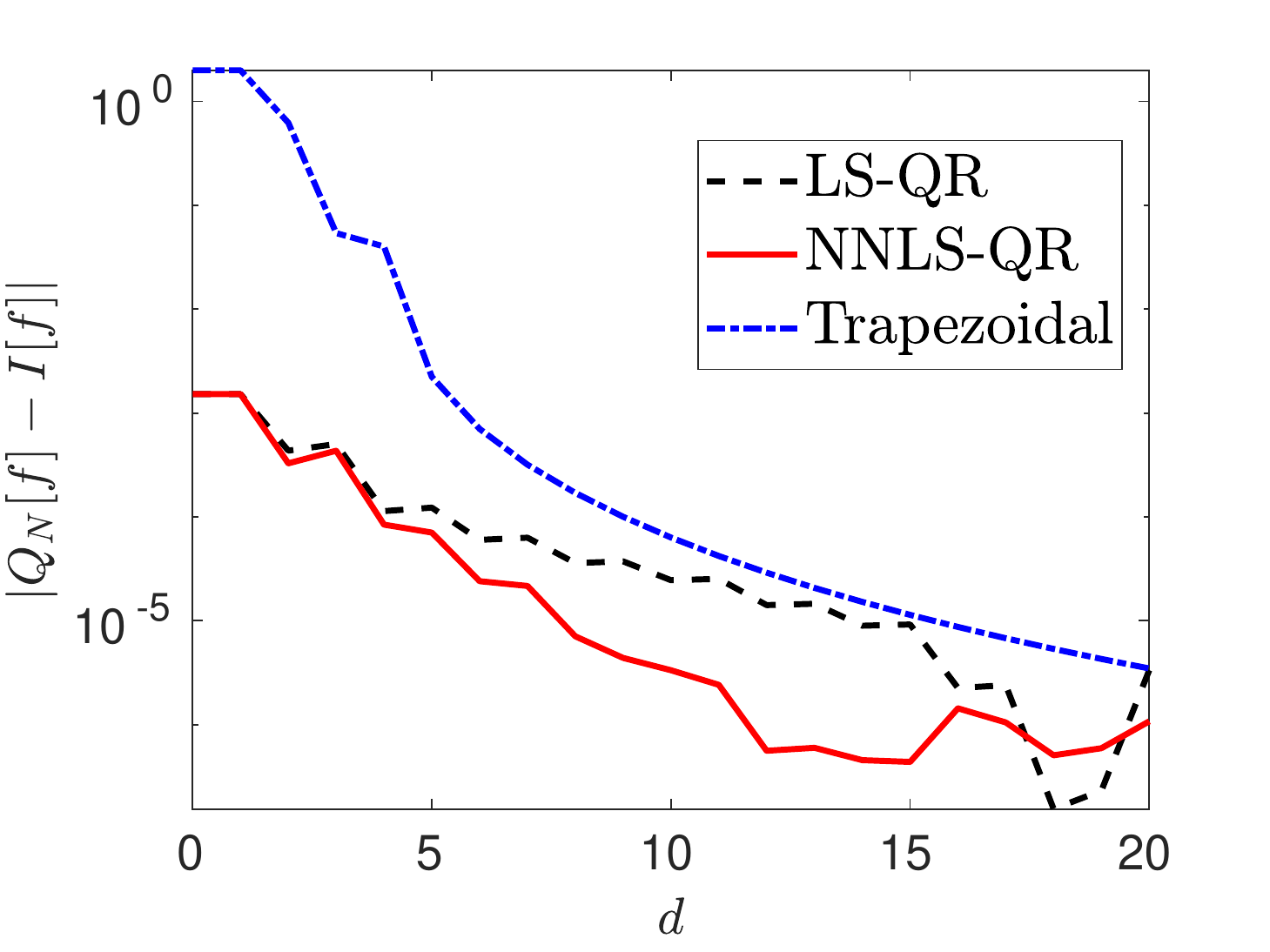}
    \caption{$f(x) = |x|^3 $ \& $\omega(x) = \cos(20 \pi x)$}
    \label{fig:accuracy-d-eq-o2-f1}
  \end{subfigure}%
  \\ 
  \begin{subfigure}[b]{0.4\textwidth}
    \includegraphics[width=\textwidth]{%
      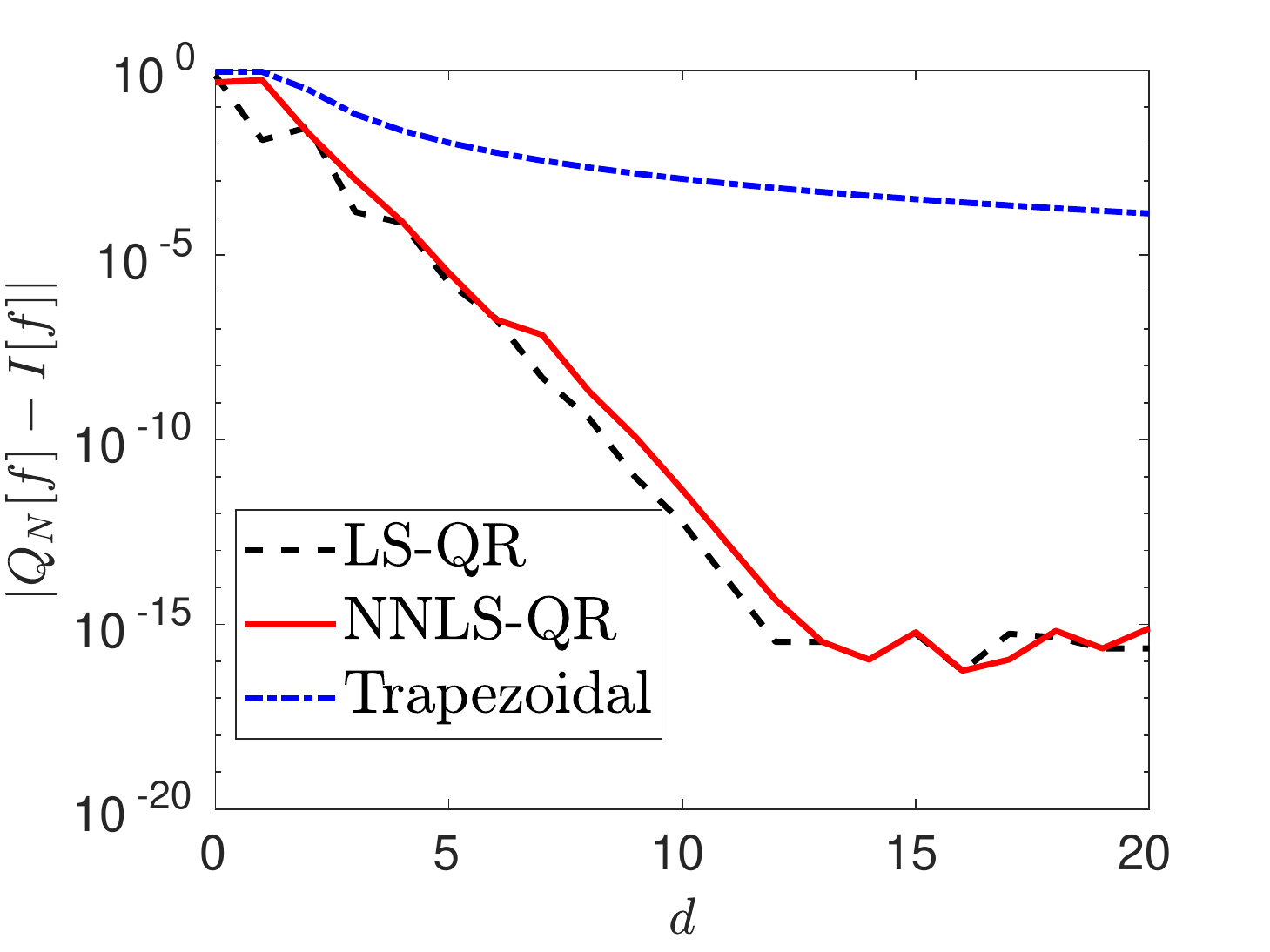}
    \caption{$f(x) = e^x $ \& $\omega(x) = x\sqrt{1-x^3}$}
    \label{fig:accuracy-d-eq-o1-f2}
  \end{subfigure}%
  ~
  \begin{subfigure}[b]{0.4\textwidth}
    \includegraphics[width=\textwidth]{%
      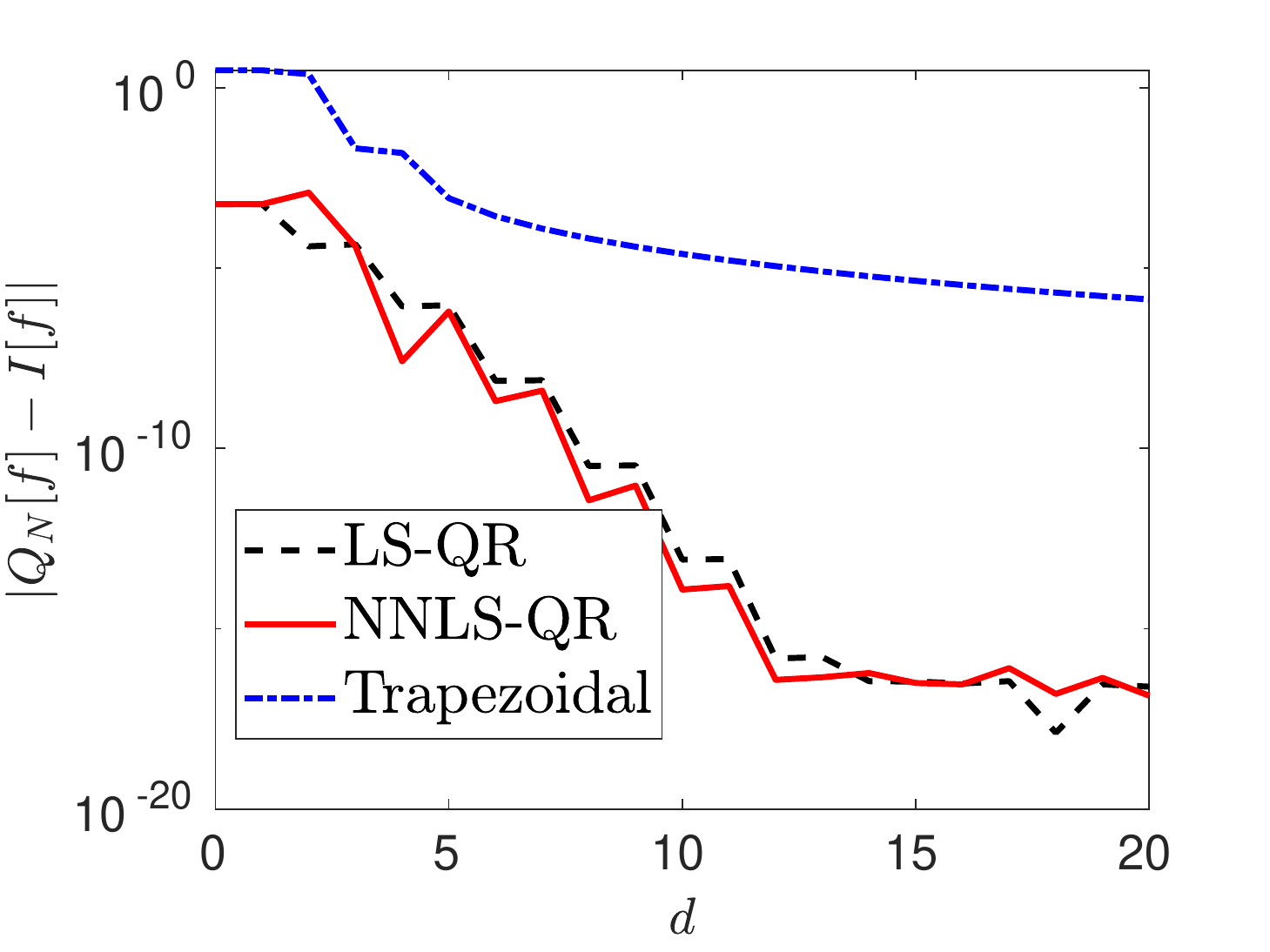}
    \caption{$f(x) = e^x $ \& $\omega(x) = \cos(20 \pi x)$}
    \label{fig:accuracy-d-eq-o2-f2}
  \end{subfigure}%
  \caption{Errors for the generalized composite trapezoidal rule, the LS-QR, 
and NNLS-QR on equidistant points.}
  \label{fig:accuracy-d-eq}
\end{figure}

Figure \ref{fig:accuracy-d-eq} illustrates the results of the LS-QR, the NNLS-QR, and the generalized composite 
trapezoidal rule for both test functions and weight functions as before on equidistant points. 
We observe that the LS-QR and the NNLS-QR provide more accurate results than the generalized composite 
trapezoidal rule in all cases. These are up to $10^{12}$ times more accurate.
Yet, in that respect, it should be noted that the product $f \omega$ is either not analytic or not periodic for 
the weight functions $\omega$ and test functions $f$ considered here. 
Otherwise, i.\,e., if $f \omega$ was analytic and periodic, the composite trapezoidal rule would converge 
geometrically; see \cite{trefethen2014exponentially}.

\begin{figure}[!htb]
  \centering
  \begin{subfigure}[b]{0.4\textwidth}
    \includegraphics[width=\textwidth]{%
      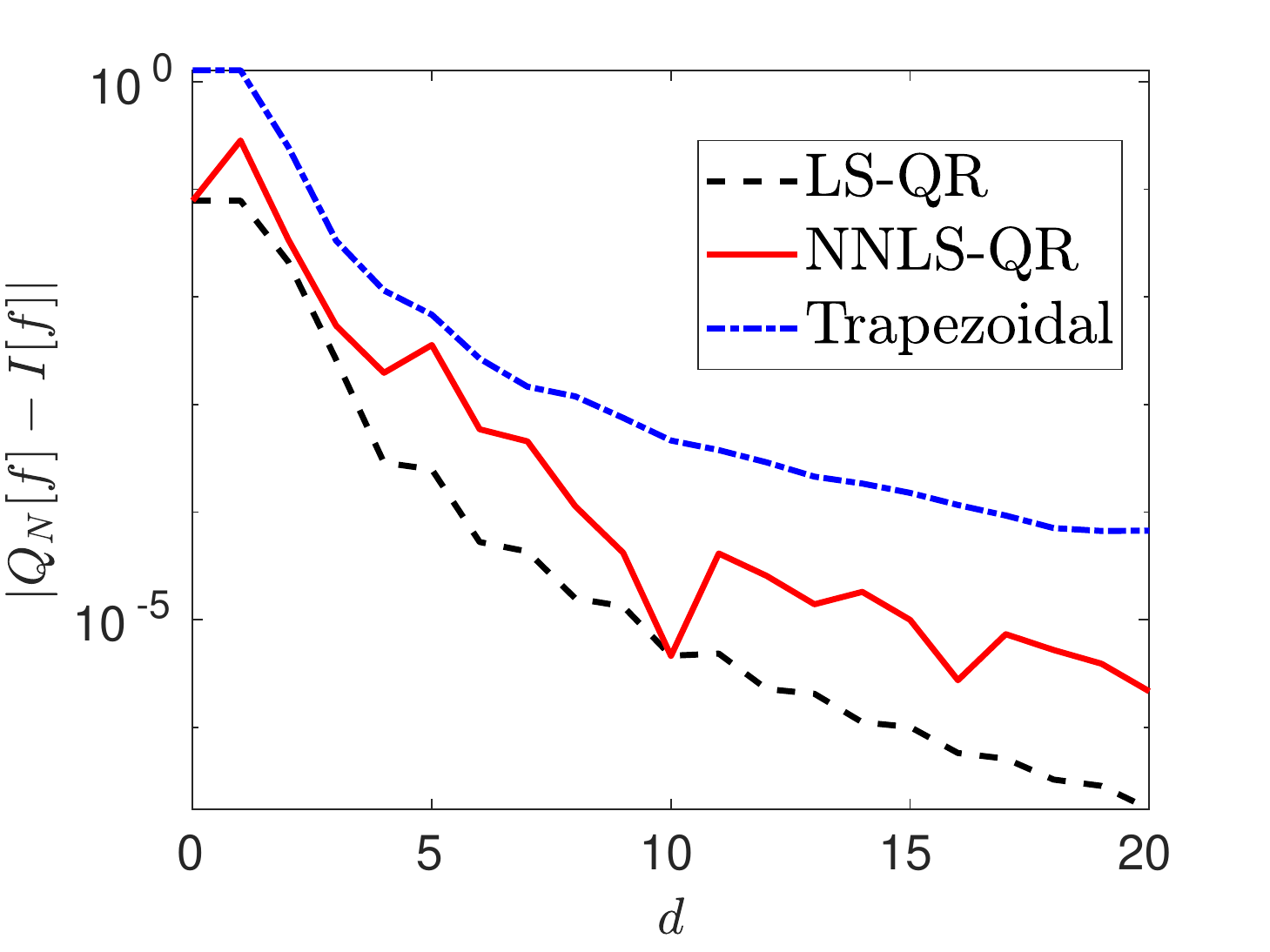}
    \caption{$f(x) = |x|^3 $ \& $\omega(x) = x\sqrt{1-x^3}$}
    \label{fig:accuracy-d-sc-o1-f1}
  \end{subfigure}%
  ~
  \begin{subfigure}[b]{0.4\textwidth}
    \includegraphics[width=\textwidth]{%
      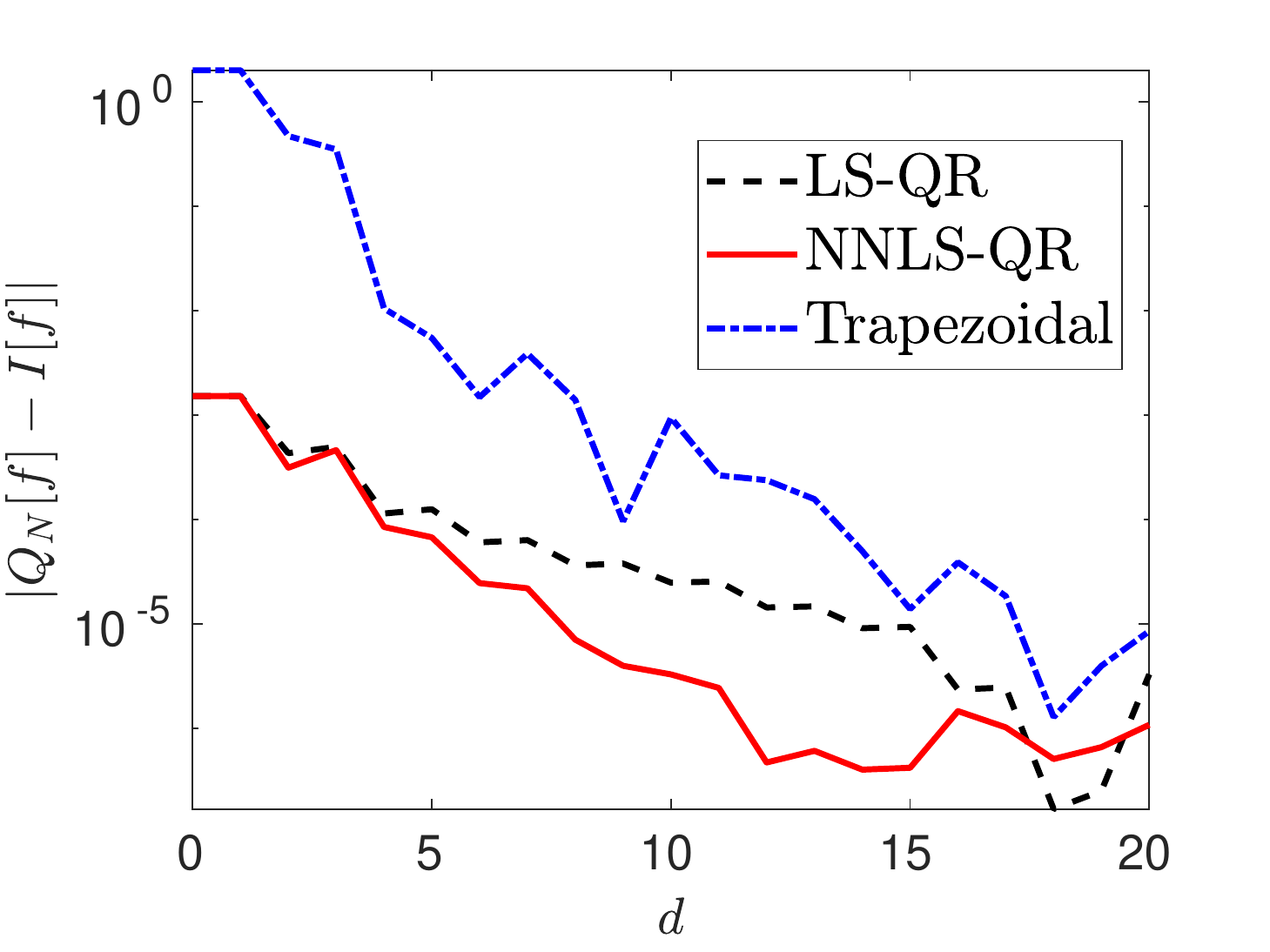}
    \caption{$f(x) = |x|^3 $ \& $\omega(x) = \cos(20 \pi x)$}
    \label{fig:accuracy-d-sc-o2-f1}
  \end{subfigure}%
  \\ 
  \begin{subfigure}[b]{0.4\textwidth}
    \includegraphics[width=\textwidth]{%
      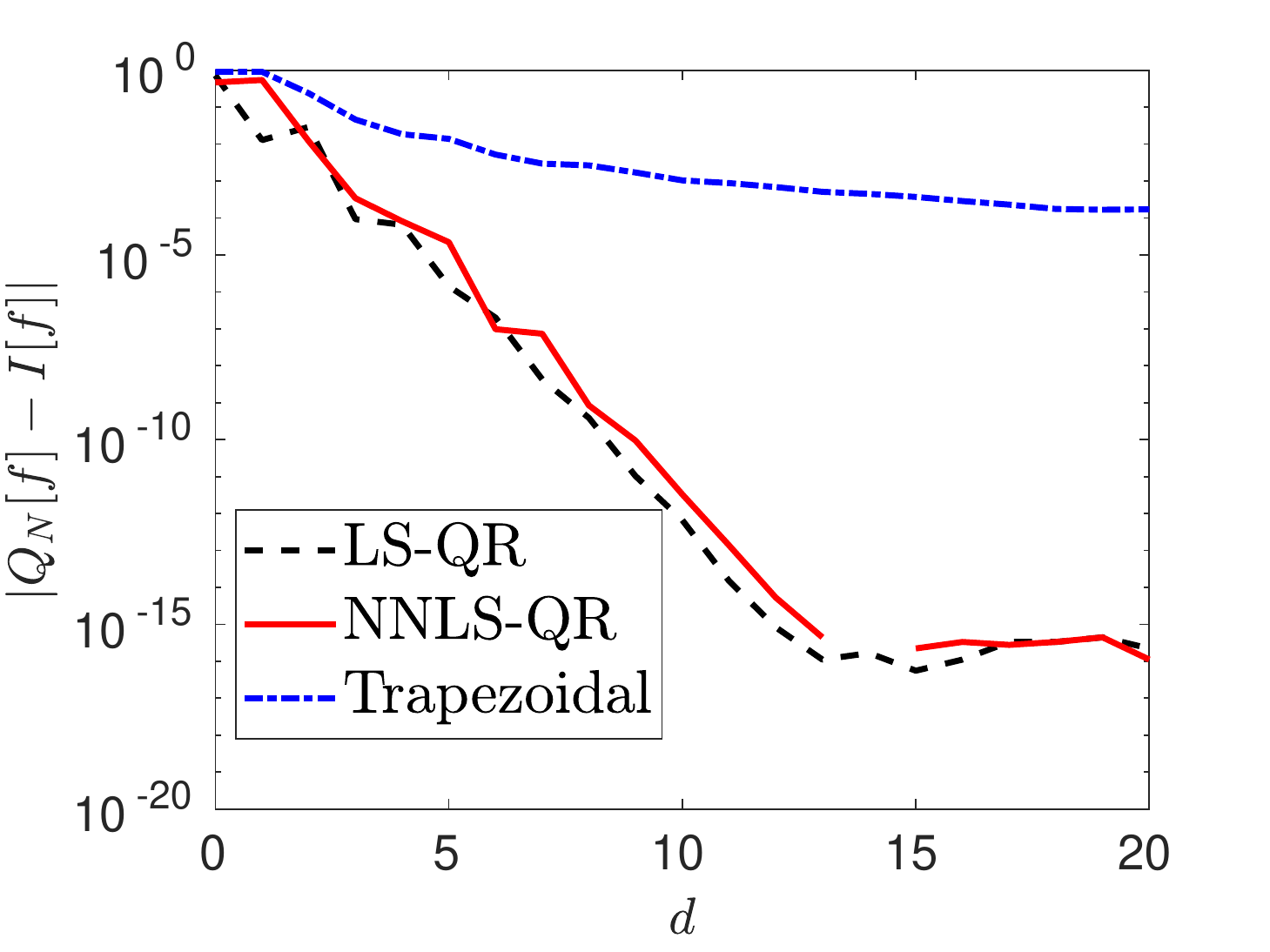}
    \caption{$f(x) = e^x $ \& $\omega(x) = x\sqrt{1-x^3}$}
    \label{fig:accuracy-d-sc-o1-f2}
  \end{subfigure}%
  ~
  \begin{subfigure}[b]{0.4\textwidth}
    \includegraphics[width=\textwidth]{%
      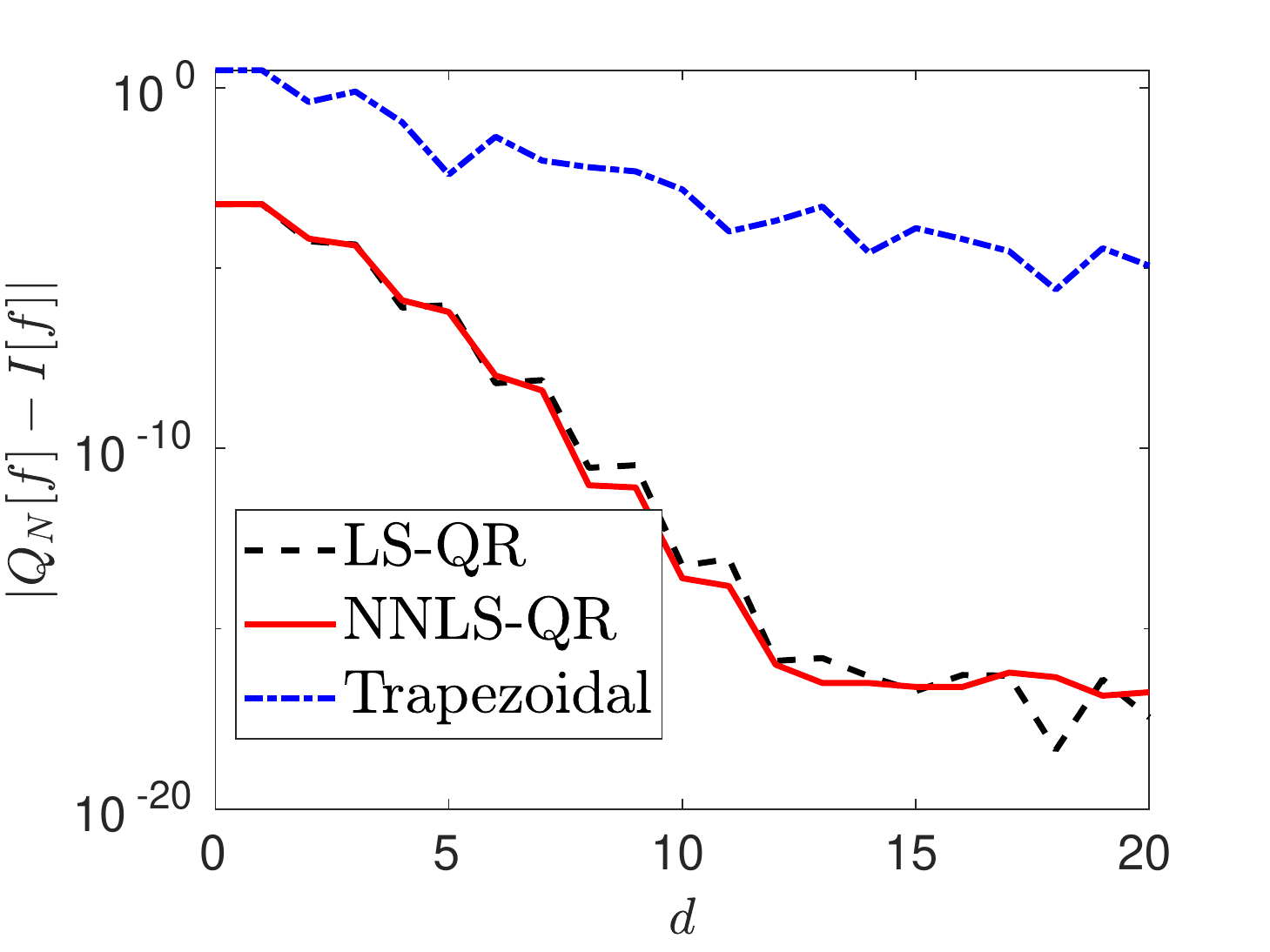}
    \caption{$f(x) = e^x $ \& $\omega(x) = \cos(20 \pi x)$}
    \label{fig:accuracy-d-sc-o2-f2}
  \end{subfigure}%
  \caption{Errors for the generalized composite trapezoidal rule, the LS-QR, 
and NNLS-QR on scattered points.}
  \label{fig:accuracy-d-sc}
\end{figure} 
 
Similar observations can be made for scattered quadrature points and are displayed in Figure \ref{fig:accuracy-d-sc}. 
Using the same set of equidistant or scattered quadrature points, the LS-QR and the NNLS-QR are able to provide highly 
accurate results even for general weight functions.

\subsection{Ratio between $d$ and $N$} 
\label{sub:ratio} 

Finally, we address the ratio between the degree of exactness $d$ and the number of equidistant quadrature points $N$ 
that is needed for stability (and exactness) for the LS-QRs (and NNLS-QRs). 
In \cite{wilson1970discrete}, Wilson showed that stability of LS-QRs for the weight function $\omega \equiv 1$ 
essentially is a $d^2$ process; that is, $N \approx C d^2$ equidistant quadrature points are needed for 
$\kappa(\vec{\omega}_N^{\mathrm{LS}}) = K_{\omega}$ to hold. 
In \cite{huybrechs2009stable} and \cite[Chapter 4]{glaubitz2020shock}, the same ratio between $d$ and $N$ has been observed for 
other positive weight functions, including $\omega(x) = 1-x^2$ and $\omega(x) = \sqrt{1-x^2}$. 
Here, we demonstrate that a similar ratio also holds for more general weight functions. 

\renewcommand{\arraystretch}{1.3}
\begin{table}[htb]
  \centering
  \begin{tabular}{c c c c c c}
    \toprule 
    $\omega(x)$ & $1$ & $1-x^2$ & $\sqrt{1-x^2}$ & $x \sqrt{1-x^3}$ & $\cos( 20 \pi x )$ \\ \hline 
    $s$ & 1.65 & 1.45 & 1.56 & 1.63 & 1.94 \\ 
    $C$ & 0.22 & 0.32 & 0.25 & 0.26 & 0.08 \\
    \bottomrule
    \end{tabular}
    \caption{LS fit of the parameters $C$ and $s$ in the model $N = C d^s$ for the LS-QR and different 
weight functions $\omega$.
    $N$ refers to the minimal number of equidistant quadrature points such that the quadrature weights of the LS-QR 
with degree of exactness $d$ satisfy $\kappa(\vec{\omega}_N^{\mathrm{LS}}) \leq 2 K_{\omega}$.
    }
    \label{tab:ratio_LS-QR}
\end{table}

\renewcommand{\arraystretch}{1.3}
\begin{table}[htb]
  \centering
  \begin{tabular}{c c c c c c}
    \toprule 
    $\omega(x)$ & $1$ & $1-x^2$ & $\sqrt{1-x^2}$ & $x \sqrt{1-x^3}$ & $\cos( 20 \pi x )$ \\ \hline 
    $s$ & 1.76 & 1.66 & 1.70 & 1.66 & 1.68 \\ 
    $C$ & 0.19 & 0.30 & 0.35 & 0.41 & 0.27 \\
    \bottomrule
    \end{tabular}
    \caption{LS fit of the parameters $C$ and $s$ in the model $N = C d^s$ for the NNLS-QR and different 
weight functions $\omega$.
    $N$ refers to the minimal number of equidistant quadrature points such that the quadrature weights of the NNLS-QR 
with approximate degree of exactness $d$ satisfy $\kappa(\vec{\omega}_N^{\mathrm{NNLS}}) \leq 2 K_{\omega}$ as well as 
$|| A \vec{\omega}_N^{\mathrm{NNLS}} - \vec{m} ||_2 \leq 10^{-14}$.
    }
    \label{tab:ratio_NNLS-QR}
\end{table}

Table \ref{tab:ratio_LS-QR} lists the results of a LS fit of the parameters $C$ and $s$ in the model $N = C d^s$ for the LS-QR and different weight functions $\omega$. 
Here, $N$ refers to the minimal number of equidistant quadrature points such that the quadrature weights of the LS-QR 
with degree of exactness $d$ satisfy $\kappa(\vec{\omega}_N^{\mathrm{LS}}) \leq 2 K_{\omega}$, where $d$ ranges from $0$ 
to $40$. 
Note that for general weight functions we have only been able to prove that $\kappa(\vec{\omega}_N^{\mathrm{LS}})$ is 
uniformly bounded w.\,r.\,t.\ $N$. 
Yet, in contrast to positive weight functions \cite[Chapter 4]{glaubitz2020shock}, it is not ensured that 
$\kappa(\vec{\omega}_N^{\mathrm{LS}}) \leq K_{\omega}$ holds for a sufficiently large $N$. 
In all numerical tests, we observed $\kappa(\vec{\omega}_N^{\mathrm{LS}}) \leq 2 K_{\omega}$ to hold for a sufficiently 
large $N$, however. 
In particular, this yields stability of the LS-QR. 
As a result of using $\kappa(\vec{\omega}_N^{\mathrm{LS}}) \leq 2 K_{\omega}$ instead of 
$\kappa(\vec{\omega}_N^{\mathrm{LS}}) \leq K_{\omega}$ as 
a criterion to determine $N$, we observe slightly smaller exponent parameters $s < 2$ in Table \ref{tab:ratio_LS-QR}. 
Yet---and more important---we can note from Table \ref{tab:ratio_LS-QR} that the ratios between $d$ and $N$ are 
similar for positive and general weight functions. 
Similar ratios also hold for the NNLS-QR and are reported in Table \ref{tab:ratio_NNLS-QR}. 
Here, $N$ refers to the minimal number of equidistant quadrature points such that the NNLS-QR with approximate degree of 
exactness $d$ (see \S \ref{sec:NNLS-QRs}) is stable as well as exact. 
Again, stability is checked by the condition $\kappa(\vec{\omega}_N^{\mathrm{NNLS}}) \leq 2 K_{\omega}$. 
Exactness, on the other hand, is checked by the condition 
$|| A \vec{\omega}_N^{\mathrm{NNLS}} - \vec{m} ||_2 \leq 10^{-14}$. 
We use the bound $10^{-14}$ instead of $0$ to encounter possible round-off errors in our implementation; 
see \S \ref{sub:implementation}.  
\section{Summary} 
\label{sec:summary} 

In this work, we have investigated stability of quadrature rules for integrals with general weight functions, possibly 
having mixed signs. 
Such weight functions, for instance, arise in the weak form of the Schr\"odinger equation and for (highly) oscillatory 
integrals. 
In contrast to non-negative weight functions, stability of quadrature rules for general weight functions is not 
guaranteed by non-negative-only quadrature weights anymore. 
Therefore, we have proposed to treat stable quadrature rules (for which round-off errors due to inexact 
arithmetics are uniformly bounded with respect to the number of quadrature points $N$) and sign-consistent quadrature 
rules (for which the signs of the quadrature weights match the signs of the weight function at the corresponding 
quadrature points) as two separated classes. 

Moreover, we have proposed two different procedures to construct such quadrature rules for general weight functions. 
In particular, our procedures allow us to construct stable high-order quadrature rules on equidistant and even 
scattered quadrature points. 
This is especially beneficial since in many applications it can be impractical, if not even impossible, to obtain 
data to fit known quadrature rules. 
Numerical tests demonstrate that both quadrature rules, referred to as \emph{least squares (LS-)} and 
\emph{non-negative least squares quadrature rules (NNLS-QR)}, are able to provide highly accurate results. 

Future work will focus on the extension of the proposed quadrature rules to higher dimensions. 
\section*{Appendix: Discrete Chebyshev polynomials}
\label{sec:appendix}

The discrete Chebyshev polynomials arise as a special case of the Hahn polynomials 
\cite{hahn1949orthogonalpolynome,szeg1939orthogonal}. 
Also see \cite[Chapter 18]{dlmf2020digital}.
We collect some of their properties which come in useful in \S \ref{sub:stability}. 
For $\alpha,\beta > -1$ and $k=0,\dots,N-1$, the \textit{Hahn polynomials} may be defined in terms of a  generalized 
hypergeometric series as 
\begin{equation}
\begin{aligned}
  Q_{k}(x;\alpha,\beta ,N-1) 
    & = {}_{3}F_{2}(-k,-x,k+\alpha +\beta +1;\alpha +1,-N+1;1) \\ 
    & = \sum_{j=0}^k \frac{ (-k)_j (k+\alpha+\beta+1)_j (-x)_j }{ (\alpha+1)_j (-N+1)_j} \frac{1}{j!} 
\end{aligned}
\end{equation}
on $[0,N-1]$, where we have used the \emph{Pochhammer symbol} 
\begin{equation}
  (a)_0 = 1, \quad 
  (a)_j = a(a+1) \dots (a+j-1).
\end{equation}
The Hahn polynomials are orthogonal on $[0,N-1]$ w.\,r.\,t.\ the inner product 
\begin{equation}
  \scp{f}{g}_{\rho} := \sum_{n=0}^{N-1} f(n) g(n) \rho(n)  
\end{equation}
with weight function 
\begin{equation}
  \rho(x) := \binom{x+\alpha}{x} \binom{N-1-x+\beta}{N-1-x}.
\end{equation}
They are normalized by 
\begin{equation}
  \norm{Q_k(\cdot,\alpha,\beta,N-1)}_{\rho}^2 
    = h_k 
    := \frac{(-1)^k (k+\alpha+\beta+1)_{N} (\beta+1)_k k!}{(2k+\alpha+\beta+1) (\alpha+1)_k (-N+1)_k (N-1)!}. 
\end{equation}
Further, Dette \cite{dette1995new} proved that for $\alpha+\beta > -1$ and 
\begin{equation}
  k \leq k(\alpha,\beta,N-1) 
    := -\frac{1}{2} \left( \alpha + \beta - 1 - \sqrt{(\alpha+\beta+1)(\alpha+\beta+2N-1)} \right)
\end{equation}
the Hahn polynomials are bounded by 
\begin{equation}
  \max_{x \in [0,N-1]} \left| Q_k(x,\alpha,\beta,N-1) \right| 
    \leq \max \left\{ 1, \frac{(\beta+1)_k}{(\alpha+1)_k} \right\}.
\end{equation}
Here, we choose $\alpha=\beta=0$, which results in the \textit{discrete Chebyshev polynomials} on $[0,N-1]$, and normalize 
and transform them to the interval $[a,b]$, resulting in the polynomials
\begin{equation}\label{eq:DOPs-eq}
  \varphi_k:[a,b] \to \R, \quad 
  \varphi_k(x) = \frac{1}{\sqrt{h_k}} Q_k\left( \frac{N-1}{b-a}(x-a), 0, 0 , N-1 \right). 
\end{equation}
These polynomials form a basis of DOPs w.\,r.\,t.\ the original inner product 
$\scp{\cdot}{\cdot}_{\vec{x}_N}$ in \eqref{eq:inner-prod} when the points $\vec{x}_N$ are equidistant, i.\,e., 
\begin{equation}
  x_n = b - a \frac{n-1}{N-1}, \quad n=1,\dots,N. 
\end{equation}
Further, for 
\begin{equation}\label{eq:cond-k-N}
  k \leq k(N) := \frac{1}{2} \left( 1 + \sqrt{2N-1} \right)
\end{equation}
they are bounded by 
\begin{equation}\label{eq:bound-inf}
  \max_{x \in [a,b]} \left| \varphi_k(x) \right| 
    \leq \frac{1}{\sqrt{h_k}}. 
\end{equation}
Finally, we note that 
\begin{align}
  (k+1)_{N} 
    & = \frac{(N+k)!}{k!}, \\ 
  (-N+1)_k 
    & = (-1)^k \frac{(N-1)!}{(N-k-1)!}.
\end{align}
Thus, we have 
\begin{equation}\label{eq:h_k}
  h_k 
    = \frac{(-1)^k (k+1)_{N} k!}{(2k+1)(-N+1)_k (N-1)!} 
    = \frac{(N+k)! (N-k-1)!}{(2k+1) (N-1)! (N-1)!}
\end{equation}
for $\alpha=\beta=0$.

\section*{Acknowledgements}
The author would like to thank the Max Planck Institute for Mathematics (MPIM) Bonn for wonderful working conditions.
Furthermore, this work is supported by the German Research Foundation (DFG, Deutsche Forschungsgemeinschaft) under Grant SO 363/15-1.

\bibliographystyle{abbrv}
\bibliography{literature}

\begin{thebibliography}{10}

\bibitem{abramowitz1964handbook}
M.~Abramowitz and I.~A. Stegun.
\newblock {\em Handbook of Mathematical Functions: With Formulas, Graphs, and
  Mathematical Tables}.
\newblock National Bureau of Standards, Washington DC, 1964.

\bibitem{brass2011quadrature}
H.~Brass and K.~Petras.
\newblock {\em Quadrature Theory: The Theory of Numerical Integration on a
  Compact Interval}.
\newblock Number 178 in Math. Surveys and Monogr. American Mathematical
  Society, Providence, RI, 2011.

\bibitem{clenshaw1960method}
C.~W. Clenshaw and A.~R. Curtis.
\newblock A method for numerical integration on an automatic computer.
\newblock {\em Numer. Math.}, 2(1):197--205, 1960.

\bibitem{davis2007methods}
P.~J. Davis and P.~Rabinowitz.
\newblock {\em Methods of Numerical Integration}.
\newblock Courier Corporation, North Chelmsford, MA, 2007.

\bibitem{dette1995new}
H.~Dette.
\newblock New bounds for {H}ahn and {K}rawtchouk polynomials.
\newblock {\em SIAM Journal on Mathematical Analysis}, 26(6):1647--1659, 1995.

\bibitem{fejer1933infinite}
L.~F{\'e}jer.
\newblock On the infinite sequences arising in the theories of harmonic
  analysis, of interpolation, and of mechanical quadratures.
\newblock {\em Bull. Amer. Math. Soc.}, 39(8):521--534, 1933.

\bibitem{gautschi1968construction}
W.~Gautschi.
\newblock Construction of {G}auss--{C}hristoffel quadrature formulas.
\newblock {\em Math. Comp.}, 22(102):251--270, 1968.

\bibitem{gautschi1997numerical}
W.~Gautschi.
\newblock {\em Numerical Analysis}.
\newblock Springer Science \& Business Media, New York, NY, 1997.

\bibitem{gautschi2004orthogonal}
W.~Gautschi.
\newblock {\em Orthogonal Polynomials: Computation and Approximation}.
\newblock Oxford University Press, Oxford, UK, 2004.

\bibitem{gelb2008discrete}
A.~Gelb, R.~B. Platte, and W.~S. Rosenthal.
\newblock The discrete orthogonal polynomial least squares method for
  approximation and solving partial differential equations.
\newblock {\em Commun. Computat. Phys.}, 3(3):734--758, 2008.

\bibitem{glaubitz2020shock}
J.~Glaubitz.
\newblock {\em Shock Capturing and High-Order Methods for Hyperbolic
  Conservation Laws}.
\newblock Logos Verlag Berlin, Berlin, Germany, 2020.

\bibitem{glaubitz2020stable}
J.~Glaubitz and P.~{\"O}ffner.
\newblock Stable discretisations of high-order discontinuous {G}alerkin methods
  on equidistant and scattered points.
\newblock {\em Appl. Numer. Math.}, 151:98--118, 2020.

\bibitem{golub2012matrix}
G.~H. Golub and C.~F. Van~Loan.
\newblock {\em Matrix Computations}.
\newblock John Hopkins University Press, Baltimore, MD, 2012.

\bibitem{hahn1949orthogonalpolynome}
W.~Hahn.
\newblock {\"U}ber {O}rthogonalpolynome, die q-{D}ifferenzengleichungen
  gen{\"u}gen.
\newblock {\em Mathe. Nachr.}, 2(1-2):4--34, 1949.

\bibitem{huybrechs2009stable}
D.~Huybrechs.
\newblock Stable high-order quadrature rules with equidistant points.
\newblock {\em J. Comput. Appl. Math.}, 231(2):933--947, 2009.

\bibitem{huybrechs2009highly}
D.~Huybrechs and S.~Olver.
\newblock Highly oscillatory quadrature.
\newblock In {\em Highly Oscillatory Problems}, pages 25--50. B. Engquist, A.
  Focus, E. Hairer, and A. Iserles, eds., Cambridge University Press,
  Cambridge, UK, 2009.

\bibitem{iserles2006highly}
A.~Iserles, S.~N{\o}rsett, and S.~Olver.
\newblock Highly oscillatory quadrature: The story so far.
\newblock In {\em Numerical Mathematics and Advanced Applications}, pages
  97--118. Springer, Berlin, 2006.

\bibitem{iserles2004quadrature}
A.~Iserles and S.~P. N{\o}rsett.
\newblock On quadrature methods for highly oscillatory integrals and their
  implementation.
\newblock {\em BIT}, 44(4):755--772, 2004.

\bibitem{krylov2006approximate}
V.~I. Krylov and A.~H. Stroud.
\newblock {\em Approximate Calculation of Integrals}.
\newblock Courier Corporation, North Chelmsford, MA, 2006.

\bibitem{lawson1995solving}
C.~L. Lawson and R.~J. Hanson.
\newblock {\em Solving Least Squares Problems}.
\newblock SIAM, Philadelphia, PA, 1995.

\bibitem{dlmf2020digital}
F.~W.~J. Olver, A.~B. {Olde Daalhuis}, D.~W. Lozier, B.~I. Schneider, R.~F.
  Boisvert, C.~W. Clark, B.~R. Miller, B.~V. Saunders, H.~S. Cohl, and M.~A.
  McClain.
\newblock {NIST} {D}igital {L}ibrary of {M}athematical {F}unctions.
\newblock Release 1.0.26, March 15, 2020, 2020.

\bibitem{szeg1939orthogonal}
G.~Szeg{\"o}.
\newblock {\em Orthogonal Polynomials}.
\newblock American Mathematical Society, Providence, RI, 1939.

\bibitem{trefethen2008gauss}
L.~N. Trefethen.
\newblock Is {G}auss quadrature better than {C}lenshaw--{C}urtis?
\newblock {\em SIAM Review}, 50(1):67--87, 2008.

\bibitem{trefethen1997numerical}
L.~N. Trefethen and D.~Bau~III.
\newblock {\em Numerical Linear Algebra}.
\newblock SIAM, Philadelphia, PA, 1997.

\bibitem{trefethen2014exponentially}
L.~N. Trefethen and J.~Weideman.
\newblock The exponentially convergent trapezoidal rule.
\newblock {\em SIAM Review}, 56(3):385--458, 2014.

\bibitem{wilson1970discrete}
M.~W. Wilson.
\newblock Discrete least squares and quadrature formulas.
\newblock {\em Math. Comp.}, 24(110):271--282, 1970.

\bibitem{wilson1970necessary}
M.~W. Wilson.
\newblock Necessary and sufficient conditions for equidistant quadrature
  formula.
\newblock {\em SIAM Journal on Numerical Analysis}, 7(1):134--141, 1970.

\end{thebibliography}

\end{document}